\documentclass[10pt]{amsart}
\usepackage{amssymb,amsmath}
\usepackage{epsfig}
\usepackage{hyperref}
\usepackage{appendix}
\newtheorem{theorem}{Theorem}[section]
\newtheorem{remm}{Remark}[section]
\newtheorem{lemma}{Lemma}[section]

\newtheorem{remark}[remm]{Remark}

\newtheorem{proposition}[theorem]{Proposition}

\numberwithin{equation}{section}
\newcommand{\rset}{{\mathbb R}}
\newcommand{\nset}{{\mathbb N}}

\newcommand{\bfdot}{\bf\dot}

\newcommand{\ken}{\ \ }

\newcommand{\ssy}{\scriptscriptstyle}

\newcommand{\half}{\frac{1}{2}}
\begin{document}
\title[]
{Fully-Discrete Finite Element Approximations for\\
a fourth-order linear stochastic parabolic equation \\
with additive space-time white noise\\
II. 2D and 3D Case}
\author[]
{Georgios T. Kossioris$^{\ddag}$ and Georgios E.
Zouraris$^{\ddag}$}
\thanks{%
$^{\ddag}$Department of Mathematics, University of Crete, GR--714
09 Heraklion, Crete, Greece.}
%
%
%
\maketitle
%
%
%
\begin{abstract}
We consider an initial- and Dirichlet boundary- value problem for
a fourth-order linear stochastic parabolic equation, in two or
three space dimensions, forced by an additive space-time white
noise.
Discretizing the space-time white noise a modeling error is
introduced and a regularized fourth-order linear stochastic
parabolic problem is obtained.
Fully-discrete approximations to the solution of the regularized
problem are constructed by using, for discretization in space, a
standard Galerkin finite element method based on $C^1$ piecewise
polynomials, and, for time-stepping, the Backward Euler method.
We derive strong a priori estimates for the modeling error and for
the approximation error to the solution of the regularized
problem.
\end{abstract}
%
%
%
\section{Introduction}\label{SECT1}
%
%
%
\subsection{The main problem}
%
%
Let $d=2$ or $3$, $T>0$, $D=(0,1)^d\subset\rset^d$ and
$(\Omega,{\mathcal F},P)$ be a complete probability space. Then we
consider an initial- and Dirichlet boundary- value problem for a
fourth-order linear stochastic parabolic equation formulated,
typically, as follows: find a stochastic function
$u:[0,T]\times{\overline D}\to\rset$ such that
\begin{equation}\label{PARAP}
\begin{gathered}
\partial_tu +\Delta^2u={\dot W}(t,x)
\quad\forall\,(t,x)\in (0,T]\times D,\\
\Delta^{m}u(t,\cdot)\big|_{\ssy\partial D}=0
\quad\forall\,t\in(0,T],
\ken m=0,1,\\
u(0,x)=0\quad\forall\,x\in D,\\
\end{gathered}
\end{equation}
a.s. in $\Omega$, where ${\dot W}$ denotes a space-time white
noise on $[0,T]\times D$ (see, e.g., \cite{Walsh86},
\cite{KXiong}).
The mild solution of the problem above (cf. \cite{Carolina1},
\cite{DapDeb1}), known as `stochastic convolution', is given by
\begin{equation}\label{Lin_MildSol}
u(t,x)=
\int_0^t\!\int_{\ssy D}G(t-s;x,y)\,dW(s,y).
\end{equation}
Here, $G(t;x,y)$ is the space-time Green kernel of the
corresponding deterministic parabolic problem: find a
deterministic function $w:[0,T]\times{\overline D}\to\rset$ such
that
\begin{equation}\label{Det_Parab}
\begin{gathered}
\partial_tw +\Delta^2w = 0
\quad\forall\,(t,x)\in (0,T]\times D,\\
\Delta^{m}w(t,\cdot)\big|_{\ssy\partial D}=0
\quad\forall\,t\in(0,T],\ken m=0,1,
\\
w(0,x)=w_0(x)\quad\forall\,x\in D,\\
\end{gathered}
\end{equation}
where $w_0$ is a deterministic initial condition.
In particular, we have
\begin{equation}\label{deter_green}
w(t,x)=\int_{\ssy D}G(t;x,y)\,w_0(y)\,dy
\quad \forall\,(t,x)\in(0,T]\times{\overline D}
\end{equation}
and
\begin{equation}\label{GreenKernel}
G(t;x,y)=\sum_{\alpha\in{\mathbb N}^d}e^{-\lambda_{\alpha}^2 t}
\,\varepsilon_{\alpha}(x)\,\varepsilon_{\alpha}(y) \quad
\forall\,(t,x,y)\in(0,T]\times{\overline D}\times{\overline D},
\end{equation}
where $\lambda_{\alpha}:=\pi^2\,|\alpha|_{\ssy\nset^d}^2$,
$|\alpha|_{\ssy\nset^d}:=\left(\sum_{i=1}^d\alpha_i^2\right)^{\frac{1}{2}}$
and $\varepsilon_{\alpha}(z):=2^{\frac{d}{2}}
\,\prod_{i=1}^d\sin(\alpha_i\,\pi\,z_i)$ for all $z\in{\overline
D}$ and $\alpha\in\nset^d$.
\subsection{The regularized problem}
Following the approach for a second order one dimensional
stochastic parabolic equation with additive space-time white noise
proposed in \cite{ANZ}, we construct below an approximate initial
and boundary value problem:
\par\noindent
\hskip 1.0truecm\vbox{\hsize 14.0truecm\noindent\strut
For $N_{\star},\,J_{\star}\in\nset$, define the mesh-lengths
$\Delta{t}:=\frac{T}{N_{\star}}$,
$\Delta{x}:=\frac{1}{J_{\star}}$, and the nodes
$t_n:=n\,\Delta{t}$ for $n=0,\dots,N_{\star}$ and
$x_j:=j\,\Delta{x}$ for $j=0,\dots,J_{\star}$. Then, we define the
sets ${\mathcal N}_{\star}:=\{1,\dots,N_{\star}\}$, ${\mathcal
J}_{\star}:=\{1,\dots,J_{\star}\}$, $T_n:=(t_{n-1},t_n)$ for
$n\in{\mathcal N}_{\star}$, $D_j:=(x_{j-1},x_j)$ for
$j\in{\mathcal J}_{\star}$, $D_{\mu}:=\prod_{i=1}^dD_{\mu_i}$ for
$\mu\in{\mathcal J}_{\star}^d$, and $S_{n,\mu}:=T_n\times D_{\mu}$
for $n\in{\mathcal N}_{\star}$ and $\mu\in{\mathcal J}_{\star}^d$.
Next, consider the fourth-order linear stochastic parabolic
problem:
\begin{equation}\label{AC2}
\begin{gathered}
\partial_t{\widehat u} +\Delta^2{\widehat u}={\widehat W}
\quad\text{\rm in}\ken (0,T]\times D,\\
\Delta^{m}{\widehat u}(t,\cdot)\big|_{\ssy\partial D}=0
\quad\forall\,t\in(0,T],\ken m=0,1,\\
{\widehat u}(0,x)=0\quad\forall\,x\in D,\\
\end{gathered}
\end{equation}
\par\noindent
a.e. in $\Omega$, where
\begin{equation*}\label{WNEQ1}
{\widehat W}(t,x) :=\tfrac{1}{\Delta{t}\,(\Delta{x})^d}
\,\sum_{n\in{\mathcal N}_{\star}}\sum_{\mu\in{\mathcal
J}_{\star}}{\mathcal X}_{\ssy S_{n,\mu}}(t,x)\,R^{n,\mu}
\quad\forall\,(t,x)\in [0,T]\times{\overline D},
\end{equation*}
\begin{equation*}
R^{n,\mu}:=\int_{\ssy S_{n,\mu}}1\;dW \quad,\forall\,n\in{\mathcal
N}_{\star}, \ \ \forall\,\mu\in{\mathcal J}_{\star}^d,
\end{equation*}
and ${\mathcal X}_{\ssy S}$ is the index function of
$S\subset[0,T]\times{\overline D}$.
\strut}
\par\noindent
The solution of the problem \eqref{AC2}, according to the standard
theory for parabolic problems (see, e.g, \cite{LMag}), has the
integral representation
\begin{equation}\label{HatUform}
\widehat{u}(t,x)= \int_0^t\!\!\!\int_{\ssy D} G(t-s;x,y)
\,{\widehat W}(s,y)\,dsdy
\quad\forall\,(t,x)\in[0,T]\times{\overline D}.
\end{equation}
\begin{remark}
The properties of the stochastic integral (see, e.g.,
\cite{Walsh86}), yield that $R^{n,\mu}\sim{\mathcal
N}(0,\Delta{t}(\Delta{x})^d)$ for all $(n,\mu)\in{\mathcal
N}_{\star}\times{\mathcal J}_{\star}^d$. Also, we observe that
${\mathbb E}[R^{n,\mu}\,R^{n',\mu'}]=0$ for
$(n,\mu)\not=(n',\mu')$. Thus, the random variables
$(R^{n,\mu})_{(n,\mu)\in{\mathcal N}_{\star}\times{\mathcal
J}_{\star}^d}$ are independent.
\end{remark}
%
%
%
\subsection{Main results of the paper}
%
%
\par
The problem \eqref{PARAP} is a linearized formulation of the
stochastic Cahn-Hilliard equation (cf. \cite{Carolina1},
\cite{DapDeb1}) which was introduced by Cook \cite{Cook} for the
investigation of phase separation in spinodal decomposition (see,
e.g., \cite{KMuth}, \cite{ERD}).
%
%
For the convergence analysis of approximation methods for
fourth-order stochastic parabolic problems driven by a space-time
white noise, we refer the reader: to \cite{Carolina2} which
considers a finite difference method for the stochastic
Cahn-Hilliard equation, and to \cite{Printems}, \cite{Erika02} and
\cite{Erika03} which consider time-stepping methods for a wide
family of evolution problems that includes \eqref{PARAP}, while
the finite element method is not among the space-discretization
techniques considered in \cite{Erika02} and \cite{Erika03}.
Also, we refer the reader to \cite{GK}, \cite{ANZ},
\cite{KloedenShot}, \cite{YubinY05}, \cite{Walsh05} and
\cite{BinLi} for the analysis of the finite element method for
second order stochastic parabolic problems.
%
%
%
\par
In the paper at hand, we extend some of our results for the 1D
case, given in \cite{KZ2008}, to the 2D and 3D case. In
particular, we consider approximations of the solution ${\widehat
u}$ of \eqref{AC2} produced by the Backward Euler time-stepping
combined with a $C^1-$finite element method, and analyze its
convergence to the mild solution of \eqref{PARAP}. This error
splits in two parts: the modeling error that appears by
approximating $u$ by ${\widehat u}$ and the numerical
approximation error for ${\widehat u}$.
The estimation of the modeling error is achieved, in
Theorem~\ref{BIG_Qewrhma1}, by obtaining the inequality
\begin{equation}\label{Model_Error}
\max_{t\in[0,T]}\left\{\,\int_{\ssy\Omega}\left(\int_{\ssy
D}|u(t,x)-{\widehat
u}(t,x)|^2\;dx\right)\,dP\,\right\}^{\frac{1}{2}}\leq\,C\,\left[\,
\epsilon^{-\frac{1}{2}}\,\,\Delta{x}^{\frac{4-d}{2}-\epsilon}
+{\Delta t}^{\frac{4-d}{8}}\,\right].
\end{equation}
%
%
Moving to the direction of building approximations of ${\widehat
u}$, we let $M\in\nset$, $(\tau_m)_{m=0}^{\ssy M}$ the nodes of a
partition of $[0,T]$, i.e. $\tau_0=0$, $\tau_{\ssy M}=T$ and
$\tau_{m-1}<\tau_m$ for $m=1,\dots,M$, and define
$\Delta_m:=(\tau_{m-1},\tau_m)$ and $k_m:=\tau_m-\tau_{m-1}$ for
$m=1,\dots,M$, and set $k_{\ssy\rm max}:=\max_{1\leq{m}\leq{\ssy
M}}k_m$; also, we let $M_h\subset H_0^1(D)\cap H^2(D)$ be a finite
element space consisting of functions which are piecewise
polynomials over a partition of $D$ in triangles or rectangulars
with maximum diameter $h$, and define a discrete biharmonic
operator $B_h:M_h\to M_h$ by
\begin{equation}\label{DLaplacianI}
(B_h\varphi,\chi)_{\ssy 0,D} =(\Delta\varphi,\Delta\chi)_{\ssy
0,D}\quad\forall\,\varphi,\chi\in M_h,
\end{equation}
and the usual $L^2(D)-$projection operator $P_h:L^2(D)\to M_h$ by
\begin{equation*}
(P_hf,\chi)_{\ssy 0,D}=(f,\chi)_{\ssy 0,D}\quad\forall\,\chi\in
M_h,\quad\forall\,f\in L^2(D).
\end{equation*}
To construct approximations to ${\widehat u}$, we  employ the
Backward Euler finite element method which begins by setting
\begin{equation}\label{FullDE1}
{\widehat U}_h^0:=0,
\end{equation}
and, then for $m=1,\dots,M$, finds ${\widehat U}_h^m\in M_h$ such
that
\begin{equation}\label{FullDE2}
{\widehat U}_h^m-{\widehat U}_h^{m-1} +k_m\,B_h{\widehat U}_h^m=
\int_{\ssy\Delta_m} P_h{\widehat W}\,ds.
\end{equation}
Estimating the numerical approximation error for ${\widehat u}$,
we derive first, in Theorem~\ref{Vivo}, the discrete in time
$L^2_t(L^2_{\ssy P}(L^2_x))$ error estimate:
\begin{equation}\label{FDEstim_Euro2a}
\left\{\sum_{m=1}^{\ssy M}k_m\,\int_{\ssy\Omega}\left(\int_{\ssy
D}\big|{\widehat U}^m_h(x)-{\widehat
u}(\tau_m,x)\big|^2\;dx\right)dP\right\}^{\frac{1}{2}}\leq\,
C\,\left[\,(k_{\ssy\max})^{\frac{4-d}{8}}
+\epsilon^{-\frac{1}{2}}\,h^{\nu-\epsilon}\,\right],
\end{equation}
where: $\nu=\nu(r,d)$ is given in \eqref{soho4} and depends on the
space dimension $d$ and a parameter $r\in\{2,3,4\}$ which is
related to the approximation properties of the finite element
spaces $M_h$ (see \eqref{Sh_H2}).
Only for the needs of the proof of \eqref{FDEstim_Euro2a}, we
introduce a space-discrete approximation ${\widehat u}_h$ of
${\widehat u}$ and analyze its convergence in the
$L^{\infty}_t(L_{\ssy P}^2(L^2_x))$ norm (see
Theorem~\ref{coconut12}).
Also, assuming that the nodes $(\tau_m)_{m=0}^{\ssy M}$ are
equidistributed with $\Delta\tau=k_m$ for $m=1,\dots,M$, we arrive
at the discrete in time $L^{\infty}_t(L^2_{\ssy P}(L^2_x))$ error
estimate (see Theorem~\ref{FFQEWR}):
\begin{equation}\label{FDEstim4}
\max_{0\leq{m}\leq{\ssy M}}\left\{\int_{\ssy\Omega}\left(
\int_{\ssy D}\big|{\widehat U}_h^m(x)-{\widehat
u}(\tau_m,x)\big|^2\;dx\right)dP\right\}^{\frac{1}{2}} \leq\,C\,
\left[\,\epsilon_1^{-\frac{1}{2}}\,\Delta\tau^{\frac{4-d}{8}-\epsilon_1}
+\epsilon_2^{-\frac{1}{2}}\,h^{\nu-\epsilon_2}\,\right].
\end{equation}
To get the estimate above, first we consider the Backward-Euler
time-discrete approximations of ${\widehat u}$ and analyze their
convergence in the discrete in time $L^{\infty}_t(L^2_{\ssy
P}(L^2_x))$ norm above (see Theorem~\ref{TimeDiscreteErr1}), and
then, we compare the Backward-Euler fully-discrete with the
Backward-Euler time-discrete approximations of ${\widehat u}$ (see
Proposition~\ref{Tigrakis}). This procedure gives the possibility
to estimate separately the space and the time discretization error
in constrast to the technique used in \cite{YubinY05} and
\cite{BinLi} for second order problems.
The reason of including the error estimate \eqref{FDEstim_Euro2a}
in addition to the stronger norm error estimate \eqref{FDEstim4},
is that the order of convergence is slightly better and it allows
a nonuniform partition of the time interval which is non standard
among references  (cf. \cite{Walsh05}, \cite{Erika02},
\cite{Erika03}, \cite{YubinY05}, \cite{BinLi}).
%
%
\par
We close the section by an overview of the paper.
Section~\ref{SECTIONCHILD} introduces notation, and recalls or
prove several results often used in the paper.
Section~\ref{SECTION2} is
dedicated to the estimation of the modeling error.
Section~\ref{SECTION3} defines the Backward Euler time-discrete
approximations of ${\widehat u}$ and analyzes its convergence.
Section~\ref{SECTION4} defines a finite element space-discrete
approximation of ${\widehat u}$ and estimates its approximation
error.
Section~\ref{SECTION44} contains the error analysis for the
Backward Euler fully-discrete approximations of ${\widehat u}$.
%
%
%
%
%
%
%
\section{Notation and Preliminaries}\label{SECTIONCHILD}
\subsection{Function spaces and operators}\label{Section2.2}
We denote by $L^2(D)$ the space of the Lebesgue measurable
functions which are square integrable on $D$ with respect to
Lebesgue's measure $dx$, provided with the standard norm
$\|g\|_{\ssy 0,D}:= \{\int_{\ssy D}|g(x)|^2\,dx\}^{\frac{1}{2}}$
for $g\in L^2(D)$. The standard inner product in $L^2(D)$ that
produces the norm $\|\cdot\|_{\ssy 0,D}$ is written as
$(\cdot,\cdot)_{\ssy 0,D}$, i.e., $(g_1,g_2)_{\ssy 0,D}
:=\int_{\ssy D}g_1(x)g_2(x)\,dx$ for $g_1$, $g_2\in L^2(D)$.
For $s\in\nset_0$, $H^s(D)$ will be the Sobolev space of functions
having generalized derivatives up to order $s$ in the space
$L^2(D)$, and by $\|\cdot\|_{\ssy s,D}$ its usual norm, i.e.
$\|g\|_{\ssy s,D}:=\bigl\{\sum_{\alpha\in{\mathbb
N}_0^d,\,|\alpha|_{\ssy\nset^d}\leq s}
\|\partial_x^{\alpha}g\|_{\ssy 0,D}^2\bigr\}^{\frac{1}{2}}$ for
$g\in H^s(D)$. Also, by $H_0^1(D)$ we denote the subspace of
$H^1(D)$ consisting of functions which vanish at the boundary
$\partial{D}$ of $D$ in the sense of trace. We note that in
$H_0^1(D)$ the, well-known, Poincar{\'e}-Friedrich inequality
holds, i.e.,
\begin{equation}\label{Poincare}
\|g\|_{\ssy 0,D}\leq\,C_{\ssy P\!F}\,\|\nabla g\|_{\ssy 0,D}
\quad\forall\,g\in H^1_0(D),
\end{equation}
where $\|\nabla{v}\|_{\ssy
0,D}:=\left(\sum_{\alpha\in\nset^d_0,\,|\alpha|_{\ssy\nset^d}=1}
\|\partial_x^{\alpha}v\|_{\ssy
0,D}^2\right)^{\frac{1}{2}}$ for $v\in  H^1(D)$.
\par
The sequence of pairs
$\{\big(\lambda_{\alpha},\varepsilon_k\big)\}_{\alpha\in{\mathbb
N}^d}$ is a solution to the eigenvalue/eigenfunction problem: find
nonzero $\varphi\in H^2(D)\cap H_0^1(D)$ and $\sigma\in\rset$ such
that $-\Delta\varphi=\sigma\,\varphi$ in $D$. Since
$(\varepsilon_{\alpha})_{\alpha\in{\mathbb N}^d}$ is a complete
$(\cdot,\cdot)_{\ssy D}-$orthonormal system in $L^2(D)$, for
$s\in\rset$, a subspace ${\bfdot H}^s(D)$ of $L^2(D)$ (see
\cite{Thomee}) is defined by
\begin{equation*}
{\bfdot H}^s(D):=\left\{v\in L^2(D):\ken \sum_{\alpha\in{\mathbb
N}^d}\lambda_{\alpha}^{s} \,(v,\varepsilon_{\alpha})^2_{\ssy
0,D}<\infty\,\right\}
\end{equation*}
and provided with the norm
$\|v\|_{\ssy{\bfdot H}^s}:=\big(\,{\sum_{\alpha\in{\mathbb N}^d}}
\lambda_{\alpha}^s\,(v,\varepsilon_{\alpha})^2_{\ssy
0,D}\,\big)^{\frac{1}{2}} \quad\forall\,v\in{\bfdot H}^s(D)$.
Let $m\in\nset_0$. It is well-known (see \cite{Thomee}) that
\begin{equation}\label{dot_charact}
{\bfdot H}^m(D)=\big\{\,v\in H^m(D):
\quad\Delta^{i}v\left|_{\ssy\partial D}\right.=0 \quad\text{\rm
if}\ken 0\leq{i}<\tfrac{m}{2}\,\big\}
\end{equation}
and there exist constants $C_{m,{\ssy A}}$ and $C_{m,{\ssy B}}$
such that
\begin{equation}\label{H_equiv}
C_{m,{\ssy A}}\,\|v\|_{\ssy m,D} \leq\|v\|_{\ssy{\bfdot H}^m}
\leq\,C_{m,{\ssy B}}\,\|v\|_{\ssy m,D}\quad \forall\,v\in{\bfdot
H}^m(D).
\end{equation}
Also, we define on $L^2(D)$ the negative norm $\|\cdot\|_{\ssy -m,
D}$ by
\begin{equation*}
\|v\|_{\ssy -m, D}:=\sup\Big\{ \tfrac{(v,\varphi)_{\ssy 0,D}}
{\|\varphi\|_{\ssy m,D}}:\quad \varphi\in{\bfdot H}^m(D)
\ken\text{\rm and}\ken\varphi\not=0\Big\} \quad\forall\,v\in
L^2(D),
\end{equation*}
for which, using \eqref{H_equiv}, it is easy to conclude that
there exists a constant $C_{-m}>0$ such that
\begin{equation}\label{minus_equiv}
\|v\|_{\ssy -m,D}\leq\,C_{-m}\,\|v\|_{{\bfdot H}^{-m}}
\quad\forall\,v\in L^2(D).
\end{equation}
\par
Let ${\mathbb L}_2=(L^2(D),(\cdot,\cdot)_{\ssy 0,D})$ and
${\mathcal L}({\mathbb L}_2)$ be the space of linear, bounded
operators from ${\mathbb L}_2$ to ${\mathbb L}_2$. We say that, an
operator $\Gamma\in {\mathcal L}({\mathbb L}_2)$ is {\sl
Hilbert-Schmidt}, when $\|\Gamma\|_{\ssy\rm
HS}:=\left\{\sum_{k=1}^{\infty} \|\Gamma\varepsilon_k\|^2_{\ssy
0,D}\right\}^{\half}<+\infty$, where $\|\Gamma\|_{\ssy\rm HS}$ is
the so called Hilbert-Schmidt norm of $\Gamma$.
We note that the quantity $\|\Gamma\|_{\ssy\rm HS}$ does not
change when we replace $\{\varepsilon_k\}_{k=1}^{\infty}$ by
another complete orthonormal system of ${\mathbb L}_2$.
It is well known (see, e.g., \cite{DunSch}) that an operator
$\Gamma\in{\mathcal L}({\mathbb L}_2)$ is Hilbert-Schmidt iff
there exists a measurable function $g:D\times D\rightarrow{\mathbb
R}$ such that $\Gamma[v](\cdot)=\int_{\ssy D}g(\cdot,y)\,v(y)\,dy$
for $v\in L^2(D)$, and then, it holds that
%
%
\begin{equation}\label{HSxar}
\|\Gamma\|_{\ssy\rm HS} =\left(\int_{\ssy D}\!\int_{\ssy
D}g^2(x,y)\,dxdy\right)^{\half}.
\end{equation}
Let ${\mathcal L}_{\ssy\rm HS}({\mathbb L}_2)$ be the set of
Hilbert Schmidt operators of ${\mathcal L}({\mathbb L}^2)$ and
$\Phi:[0,T]\rightarrow {\mathcal L}_{\ssy\rm HS}({\mathbb L}_2)$.
Also, for a random variable $X$, let ${\mathbb E}[X]$ be its
expected value, i.e., ${\mathbb E}[X]:=\int_{\ssy\Omega}X\,dP$.
Then, the It{\^o} isometry property for stochastic integrals,
which we will use often in the paper, reads
\begin{equation}\label{Ito_Isom}
{\mathbb E}\left[\Big\|\int_0^{\ssy T}\Phi\,dW\Big\|_{\ssy
0,D}^2\right] =\int_0^{\ssy T}\|\Phi(t)\|_{\ssy\rm HS}^2\,dt.
\end{equation}
\par
For later use, we introduce the projection operator
${\widehat\Pi}:L^2((0,T)\times D) \rightarrow L^2((0,T)\times D)$
defined by
\begin{equation}\label{Defin_L2}
{\widehat \Pi}(g;\cdot)\left|_{\ssy S_{n,\mu}}\right.
:=\tfrac{1}{{\Delta t}\,{\Delta x}^d}\, \int_{\ssy
S_{n,\mu}}g(t,x)\;dtdx, \quad\forall\,n\in{\mathcal N}_{\star},\ \
\forall\,\mu\in{\mathcal J}_{\star}^d,
\end{equation}
for $g\in L^2((0,T)\times D)$, which has the following property:
%
%
%
\begin{lemma}\label{Lhmma1}
For $g\in L^2((0,T)\times D)$, it holds that
\begin{equation}\label{WNEQ2}
\int_0^{\ssy T}\!\!\!\int_{\ssy D}\,{\widehat\Pi}(g;s,y)\,dW(s,y)
=\int_0^{\ssy T}\!\!\!\int_{\ssy D}\,{\widehat
W}(t,x)\,g(t,x)\,dtdx.
\end{equation}
\end{lemma}
%
%
%
%
%
%
%
%
%
%
%
%
%
%
%
%
%
%
\begin{proof}
To obtain \eqref{WNEQ2} we work, using \eqref{Defin_L2} and the
properties of $W$, as follows:
\begin{equation*}
\begin{split}
\int_0^{\ssy T}\!\!\!\int_{\ssy D}{\widehat\Pi}(g;s,y)\,dW(s,y)
=&\tfrac{1}{\Delta{t}\,(\Delta{x})^d}\, \sum_{n\in{\mathcal
N}_{\star}}\sum_{\mu\in{\mathcal J}_{\star}^d}\Bigr(\int_{\ssy
S_{n,\mu}}g\;dtdx\Bigl)\, \Bigl(\int_0^{\ssy T}\!\!\!\int_{\ssy D}
{\mathcal X}_{\ssy S_{n,\mu}}(s,y)\,dW(s,y)\Bigr)\\
=&\tfrac{1}{\Delta{t}\,(\Delta{x})^d}\, \sum_{n\in{\mathcal
N}_{\star}}\sum_{\mu\in{\mathcal J}_{\star}^d} \Bigr(\int_{\ssy
S_{n,\mu}}g(t,x)\;dtdx\Bigl)\,R^{n,\mu}\\
=&\tfrac{1}{\Delta{t}\,(\Delta{x})^d}\, \sum_{n\in{\mathcal
N}_{\star}} \sum_{\mu\in{\mathcal J}_{\star}^d}\int_0^{\ssy
T}\!\!\!\int_{\ssy D} g(t,x) \,{\mathcal X}_{\ssy S_{n,\mu}}(t,x)
\,R^{n,\mu}\,dtdx\\
=&\int_0^{\ssy T}\!\!\!\int_{\ssy D}g(t,x) \,{\widehat
W}(t,x)\,dtdx.
\end{split}
\end{equation*}
\end{proof}
%
%
%
%
%
%
%
%
%
%
%
%
\par
We close this section, by stating some asymptotic bounds for
series that will often appear in the rest of the paper and for a
proof of them we refer the reader to Appendix~\ref{Gatos_Rex_1}
and Appendix~\ref{Gatos_Rex_2}.
%
%
\begin{lemma}\label{Series_Asym_1}
Let $d\in\{1,2,3\}$ and $c_{\star}>0$. Then, there exists a
constant $C>0$ that depends on $c_{\star}$ and $d$, such that
\begin{equation}\label{ASYM_1}
\sum_{\alpha\in\nset^d}
|\alpha|_{\ssy\nset^d}^{-(d+c_{\star}\epsilon)}
\leq\,C\,\epsilon^{-1}\quad\forall\,\epsilon\in(0,2].
\end{equation}
\end{lemma}
%
%
%
%
%
%
%
%
%
\begin{lemma}\label{Series_Asym_2}
Let $d\in\{2,3\}$ and $\delta>0$. Then there exists a constant
$C>0$ which is independent of $\delta$, such that
\begin{equation}\label{ASYM_2}
\sum_{\alpha\in\nset^d}
\tfrac{1-e^{-\lambda_{\alpha}^2\delta}}{\lambda_{\alpha}^2}
\leq\,C\,\,\,p_d(\delta^{\frac{1}{4}})\,\,\,\delta^{\frac{4-d}{4}},
\end{equation}
where $p_d(s):=1+\sum_{i=1}^ds^i$.
\end{lemma}
%
%
%
%
%
%
%
%
%
%
%
%
%
%
%
%
%
%
%
%
\subsection{Linear elliptic and parabolic
operators}\label{SECTION31}
%
For given $f\in L^2(D)$ let $v_{\ssy E}\in H^2(D)\cap H_0^1(D)$ be
the solution of the boundary value problem
\begin{equation}\label{ElOp1}
\Delta v_{\ssy E} = f \quad\text{\rm in}\ken D,
\end{equation}
and $T_{\ssy E}:L^2(D)\rightarrow H^2(D)\cap H^1_0(D)$ be its
solution operator, i.e. $T_{\ssy E}f:=v_{\ssy E}$, which has the
property
\begin{equation}\label{ElReg1}
\|T_{\ssy E}f\|_{\ssy m,D}\leq \,C_{\ssy E} \,\|f\|_{\ssy m-2, D},
\quad\forall\,f\in H^{\max\{0,m-2\}}(D), \ken\forall\,m\in{\mathbb
N}_0.
\end{equation}
Also, for $f\in L^2(D)$ let $v_{\ssy B}\in H^4(D)$ be the solution
of the following biharmonic boundary value problem
\begin{gather}\label{ElOp2}
\Delta^2v_{\ssy B}=f\quad\text{\rm in}\ken D,\\
\Delta^m v_{\ssy B}\big|_{\ssy\partial D}=0,\quad m=0,1,
\end{gather}
and  $T_{\ssy B}:L^2(D)\rightarrow{\bfdot H}^4(D)$ be the solution
operator of \eqref{ElOp2}, i.e. $T_{\ssy B}f:=v_{\ssy B}$, which
satisfies
\begin{equation}\label{ElBihar2}
\|T_{\ssy B}f\|_{\ssy m,D}\leq \,C\,\|f\|_{\ssy m-4, D},
\quad\forall\,f\in H^{\max\{0,m-4\}}(D), \ken\forall\,m\in{\mathbb
N}_0.
\end{equation}
Due to the type of boundary conditions of \eqref{ElOp2}, we
conclude that
\begin{equation}\label{tiger007}
T_{\ssy B}f= T_{\ssy E}^2f,\quad\forall\,f\in L^2(D),
\end{equation}
which, easily, yields
\begin{equation}\label{TB-prop1}
(T_{\ssy B}v_1,v_2)_{\ssy 0,D} =(T_{\ssy E}v_1,T_{\ssy
E}v_2)_{\ssy 0,D} \quad\forall\,v_1,v_2\in L^2(D).
\end{equation}
%
%
\par
Letting $({\mathcal S}(t)w_0)_{t\in[0,T]}$ be the standard
semigroup notation for the solution $w$ of \eqref{Det_Parab}, we
can easily establish the following property (see, e.g.,
\cite{Thomee}, \cite{Pazy}): for $\ell\in{\mathbb N}_0$, $\beta$,
$p\in{\mathbb R}_0^{+}$ and $q\in[0,p+4\ell]$ there exists a
constant $C>0$ such that:
\begin{equation}\label{Reggo3}
\int_{t_a}^{t_b}(\tau-t_a)^{\beta}\,
\big\|\partial_t^{\ell}{\mathcal S}(\tau)w_0 \big\|_{\ssy {\bfdot
H}^p}^2\,d\tau \leq\,C\, \|w_0\|^2_{\ssy {\bfdot
H}^{p+4\ell-2\beta-2}} \quad\forall\,t_b>t_a\ge0,
\ken\forall\,w_0\in{\bfdot H}^{p+4\ell-2\beta-2}(D).
\end{equation}
%
%
%
\subsection{Discrete spaces and operators}
%
%
For $r\in\{2,3,4\}$, we consider a finite element space
$M_h\subset H_0^1(D)\cap H^2(D)$ consisting of functions which are
piecewise polynomials over a partition of $D$ in triangles or
rectangulars with maximum mesh-length $h$.  We assume that the
space $M_h$ has the following approximation property
\begin{equation}\label{Sh_H2}
\inf_{\chi\in M_h} \|v-\chi\|_{\ssy 2,D}
\leq\,C\,h^{r-1}\,\|v\|_{\ssy r+1,D} \quad\,\forall\,v\in
H^{r+1}(D)\cap H_0^1(D),
\end{equation}
which covers several classes of $C^1$ finite element spaces, for
example the tensor products of $C^1$ splines, the Argyris triangle
elements, the Hsieh-Clough-Tocher triangle elements and the Bell
triangle (cf. \cite{Cia}, \cite{BrScott}).
\par
A finite element approximation $v_{\ssy B,h}\in M_h$ of the
solution $v_{\ssy B}$ of \eqref{ElOp2} is defined by the
requirement
\begin{equation}\label{fem2}
B_hv_{\ssy B,h}=P_hf,
\end{equation}
and we denote by $T_{\ssy B,h}:L^2(D)\to M_h$ the solution
operator of \eqref{fem2}, i.e. $T_{\ssy B,h}f:=v_{\ssy
B,h}=B_h^{-1}P_hf$ for $f\in L^2(D)$. It is easy to verify that
$T_{\ssy B,h}$ is selfadjoint, i.e.,
\begin{equation}\label{adjo_Bh}
(T_{\ssy B,h}f,g)_{\ssy 0,D}=(f,T_{\ssy B,h}g)_{\ssy
0,D}\quad\forall\,f,g\in L^2(D).
\end{equation}
Also, using \eqref{fem2}, \eqref{ElOp2} and \eqref{ElBihar2} we
conclude that
\begin{equation}\label{PanwFragma1}
\begin{split}
\|\Delta T_{\ssy B,h}f\|_{\ssy 0,D}\leq&\,
\|\Delta T_{\ssy B}f\|_{\ssy 0,D}\\
\leq&\,C\,\|f\|_{\ssy -2,D}\quad\quad\forall\,f\in L^2(D).\\
\end{split}
\end{equation}
%
%
%
\par
Applying the standard theory of the finite element method (see,
e.g., \cite{Cia}, \cite{BrScott}) and using \eqref{ElBihar2}, we
get
\begin{equation}\label{C1_FEM}
\|\Delta(T_{\ssy B}f-T_{\ssy B,h}f)\|_{\ssy 0,D}
\leq\,C\,h^{r-1}\,\|f\|_{\ssy r-3,D},\quad\forall\,f\in
H^{\max\{r-3,0\}}(D),
\end{equation}
while error estimates in the $L^2(D)$ norm are obtained in the
proposition below.
%
%
\begin{proposition}\label{H2CASE}
Let $r\in\{2,3,4\}$. Then, it holds that:
\begin{equation}\label{ARA2}
\begin{split}
\|T_{\ssy B}f-T_{\ssy B,h}f\|_{\ssy 0,D} \leq\,C\left\{ \aligned
&h^5\,\|f\|_{\ssy 1,D},\quad r=4\\
&h^4\,\|f\|_{\ssy 0,D},\quad r=3,\\
&h^2\,\|f\|_{\ssy -1,D},\hskip0.15truecm r=2,\\
\endaligned\right.\quad\quad\forall\,f\in H^{\max\{r-3,0\}}(D).
\end{split}
\end{equation}
\end{proposition}
%
%
%
%
\begin{proof}
Let $f\in H^{\max\{0,r-3\}}(D)$ and $e=T_{\ssy B}f-T_{\ssy B,h}f$.
Also, we define a bilinear form $\gamma:H^2(D)\times
H^2(D)\to{\mathbb R}$ by $\gamma(v_1,v_2):=(\Delta v_1,\Delta
v_2)_{\ssy 0,D}$ for $v_1$, $v_2\in H^2(D)$.
Now, let $w_{\ssy A}$, $w_{\ssy B}\in{\bfdot H}^4(D)$ be defined
by $T_{\ssy B}\Delta e=w_{\ssy A}$ and $T_{\ssy B}e=w_{\ssy B}$.
Then, using Galerkin orthogonality, we have:
\begin{equation}\label{Basiko_Est_1}
\begin{split}
\|\nabla e\|_{\ssy 0,D}^2=&\,-\gamma(w_{\ssy A},e)_{\ssy 0,D}\\
\leq&\,\|\Delta{e}\|_{\ssy 0,D}\,\inf_{\chi\in
M_h}\|w_{\ssy A}-\chi\|_{\ssy 2,D}\\
\end{split}
\end{equation}
and
\begin{equation}\label{Basiko_Est_2}
\begin{split}
\|e\|_{\ssy 0,D}^2=&\,-\gamma(w_{\ssy A},e)_{\ssy 0,D}\\
\leq&\,\|\Delta{e}\|_{\ssy 0,D}\,\inf_{\chi\in
M_h}\|w_{\ssy B}-\chi\|_{\ssy 2,D}.\\
\end{split}
\end{equation}
\par
{\sl Case 1}: Let $r\in\{2,3\}$. Then, using \eqref{Basiko_Est_2},
\eqref{C1_FEM}, \eqref{Sh_H2} and \eqref{PanwFragma1}, we obtain
\begin{equation*}
\begin{split}
\|e\|_{\ssy 0,D}^2\leq&\,C\,h^{r-1}\,\|f\|_{\ssy r-3,D}
\,h^{r-1}\,\|w_{\ssy B}\|_{\ssy r+1,D}\\
\leq&\,C\,h^{2(r-1)}\,\|f\|_{\ssy r-3,D}\,\|e\|_{\ssy r-3,D}\\
\end{split}
\end{equation*}
which, obviously, yields \eqref{ARA2}.
\par
{\sl Case 2}: Let $r=4$. Then, combining, \eqref{Basiko_Est_2},
\eqref{Sh_H2}, \eqref{ElBihar2} and \eqref{Poincare}, we get
\begin{equation}\label{soublaki1}
\begin{split}
\|e\|_{\ssy 0,D}^2\leq&\,C\,\|\Delta{e}\|_{\ssy 0,D}
\,h^3\,\|T_{\ssy B}e\|_{\ssy 5,D}\\
\leq&\,C\,\|\Delta{e}\|_{\ssy 0,D}
\,h^3\,\|\nabla e\|_{\ssy 0,D}.\\
\end{split}
\end{equation}
Also, we observe that \eqref{Basiko_Est_1} and \eqref{ElBihar2}
yield
\begin{equation}\label{soublaki2}
\begin{split}
\|\nabla e\|_{\ssy 0,D}\leq&\,\|\Delta{e}\|_{\ssy
0,D}^{\frac{1}{2}}\,\|T_{\ssy B}\Delta{e}\|_{\ssy
2,D}^{\frac{1}{2}}\\
\leq&\,C\,\|\Delta{e}\|_{\ssy 0,D}^{\frac{1}{2}}
\,\|e\|_{\ssy 0,D}^{\frac{1}{2}}.\\
\end{split}
\end{equation}
Now, we combine \eqref{soublaki1}, \eqref{soublaki2} and
\eqref{C1_FEM} to have
\begin{equation*}
\begin{split}
\|e\|_{\ssy 0,D}^{\frac{3}{2}}\leq&\,C\,h^{3}\,\|\Delta{e}\|_{\ssy
0,D}^{\frac{3}{2}}\\
\leq&\,C\,h^{\frac{15}{2}}\,\|f\|_{\ssy 1,D}^{\frac{3}{2}},\\
\end{split}
\end{equation*}
which obviously leads to \eqref{ARA2} for $r=4$.
\end{proof}
%
%
%
%
%
\section{An Estimate for the Modeling Error}\label{SECTION2}
Here, we derive an $L^{\infty}_t(L^2_{\ssy P}(L^2_x))$ bound for
the modeling error $u-{\widehat u}$, in terms of $\Delta{t}$ and
$\Delta{x}$.
%
%
%
%
\begin{theorem}\label{BIG_Qewrhma1}
Let $u$ and ${\widehat u}$ be defined, respectively, by
\eqref{Lin_MildSol} and \eqref{HatUform}. Then, there exists a
real constant $C>0$, independent of $T$, $\Delta{t}$ and
$\Delta{x}$, such that
\begin{equation}\label{ModelError}
\max_{[0,T]}\left\{{\mathbb E}\left[\|u-{\widehat u}\|_{\ssy
0,D}^2\right] \right\}^{\half}
\leq\,C\,\left[\,(p_d(\Delta{t}^{\frac{1}{4}}))^{\frac{1}{2}}
\,\,\,\Delta{t}^\frac{4-d}{8} +\epsilon^{-\frac{1}{2}}\,\,\,
\Delta{x}^{\frac{4-d}{2}-\epsilon}\,\,\,\right]
\quad\forall\,\epsilon\in(0,\tfrac{4-d}{2}].
\end{equation}
\end{theorem}
%
%
%
%
%
%
\begin{proof}
Using \eqref{Lin_MildSol} and \eqref{HatUform}, we conclude that
\begin{equation}\label{corv0}
u(t,x)-{\widehat u}(t,x)=\int_0^{\ssy T}\!\!\!\int_{\ssy D}
\big[{\mathcal X}_{(0,t)}(s)\,G(t-s;x,y) -{\widetilde
G}(t,x;s,y)\big]\,dW(s,y) \quad\forall\,(t,x)\in[0,T]\times
{\overline D},
\end{equation}
where $\widetilde{G}:
(0,T)\times{D}\rightarrow L^2((0,T)\times D)$ given by
\begin{equation*}
{\widetilde G}(t,x;\cdot)\Big|_{\ssy S_{n,\mu}} \equiv
\tfrac{1}{\Delta{t}\,(\Delta{x})^d} \int_{\ssy S_{n,\mu}}
{\mathcal X}_{(0,t)}(s')\,G(t-s';x,y')\,\,ds'dy', \quad
\forall\,n\in{\mathcal N}_{\star}, \ken\forall\,\mu\in{\mathcal
J}_{\star}^d.
\end{equation*}
\par
Let $\Theta:=\left\{{\mathbb E}\left[ \|u_{\ssy L}-{\widehat
u}_{\ssy L}\|_{\ssy 0,D}^2\right]\right\}^{\half}$ and
$t\in(0,T]$. Using \eqref{corv0} and It{\^o} isometry
\eqref{Ito_Isom}, we obtain
\begin{equation*}
\begin{split}
\Theta(t)&= \tfrac{1}{\Delta{t}\,(\Delta{x})^d}\Bigg\{
\sum_{n\in{\mathcal N}_{\star}}\sum_{\mu\in{\mathcal
J}_{\star}^d}\int_{\ssy D}\Bigg\{\int_{\ssy S_{n,\mu}}
\Bigg[\int_{\ssy S_{n,\mu}}
\Big[{\mathcal X}_{(0,t)}(s)\,G(t-s;x,y)\\
&\hskip6.5truecm -{\mathcal
X}_{(0,t)}(s')\,G(t-s';x,y')\Big]\,ds'dy'
\Bigg]^2\,dsdy\Bigg\}\,dx\Bigg\}^{\frac{1}{2}}.\\
\end{split}
\end{equation*}
Now, we introduce the splitting
\begin{equation}\label{corv1}
\Theta(t) \leq\,\Theta_{\ssy A}(t)
+\Theta_{\ssy B}(t),
\end{equation}
where
\begin{equation*}
\begin{split}
\Theta_{\ssy A}(t):= \tfrac{1}{\Delta{t}\,(\Delta{x})^d}\Bigg\{
\sum_{n\in{\mathcal N}_{\star}} \sum_{\mu\in{\mathcal
J}_{\star}^d}\int_{\ssy D} \Bigg\{\int_{\ssy S_{n,\mu}}\Big[
\int_{\ssy S_{n,\mu}}&{\mathcal X}_{(0,t)}(s)
\Big[\,G(t-s;x,y)\\
&-G(t-s;x,y')\Big]\,ds'dy'
\Big]^2\,dsdy\Bigg\}\,dx\Bigg\}^{\frac{1}{2}}\\
\end{split}
\end{equation*}
and
\begin{equation*}
\begin{split}
\Theta_{\ssy B}(t)= \tfrac{1}{\Delta{t}\,(\Delta{x})^d}\Bigg\{
\sum_{n\in{\mathcal N}_{\star}}\sum_{\mu\in{\mathcal
J}_{\star}^d}\int_{\ssy D}\Bigg\{\int_{\ssy S_{n,\mu}}
\Big[\int_{\ssy S_{n,\mu}}& \Big[{\mathcal
X}_{(0,t)}(s)\,G(t-s;x,y') \\
&-{\mathcal X}_{(0,t)}(s')\,G(t-s';x,y')\Big]\,ds'dy'
\Big]^2\,dsdy\Bigg\}\,dx\Bigg\}^{\frac{1}{2}}.\\
\end{split}
\end{equation*}
\par\smallskip
\par
{\tt Estimation of $\Theta_{\ssy A}(t)$}: Using
\eqref{GreenKernel} and the $(\cdot,\cdot)_{\ssy
0,D}-$orthogonality of
$(\varepsilon_{\alpha})_{\alpha\in\nset^d}$, we have
\begin{equation*}
\begin{split}
\Theta^2_{\ssy A}(t) &=\,\tfrac{1}{(\Delta{x})^{2d}}
\sum_{n\in{\mathcal N}_{\star}}\sum_{\mu\in{\mathcal
J}_{\star}^d}\int_{\ssy D} \Bigg\{\int_{S_{n,\mu}}
\Bigg[\int_{\ssy D_{\mu}}{\mathcal X}_{(0,t)}(s)
\,\Big[G(t-s;x,y)-G(t-s;x,y')\Big]\,dy'\Bigg]^2\,dsdy
\Bigg\}dx\\
&=\,\tfrac{1}{(\Delta{x})^{2d}} \sum_{n\in{\mathcal
N}_{\star}}\sum_{\mu\in{\mathcal
J}_{\star}^d}\Bigg\{\int_{S_{n,\mu}} \Bigg[\sum_{\alpha\in\nset^d}
{\mathcal X}_{(0,t)}(s)\,e^{-2\lambda_{\alpha}^2(t-s)}
\,\Big(\int_{\ssy
D_{\mu}}(\varepsilon_{\alpha}(y)-\varepsilon_{\alpha}(y'))\,dy'
\Big)^2 \Bigg]\,dsdy\Bigg\}\\
&=\,\tfrac{1}{(\Delta{x})^{2d}}\sum_{\alpha\in\nset^d}\Bigg\{
\sum_{n\in{\mathcal N}_{\star}}\int_{\ssy T_n} {\mathcal
X}_{(0,t)}(s)\,e^{-2\lambda_{\alpha}^2(t-s)}\,ds
\Bigg\}\Bigg\{\sum_{\mu\in{\mathcal J}_{\star}^d}\int_{\ssy
D_{\mu}} \Big(\int_{\ssy
D_{\mu}}(\varepsilon_{\alpha}(y)-\varepsilon_{\alpha}(y'))\,dy'\Big)^2
\,dy\Bigg\},\\
\end{split}
\end{equation*}
from which, using the Cauchy-Schwarz inequality, follows that
\begin{equation}\label{corv2}
\Theta^2_{\ssy A}(t)\leq \sum_{\alpha\in\nset^d} \Big(\int_0^t
e^{-2\lambda_{\alpha}^2(t-s)}\,ds\Big)\,
\Bigg[\tfrac{1}{(\Delta{x})^d}\sum_{\mu\in{\mathcal J}_{\star}^d}
\int_{\ssy D_{\mu}\times D_{\mu}}
\,\big|\varepsilon_{\alpha}(y)-\varepsilon_{\alpha}(y')\big|^2\,dy'dy
\Bigg].
\end{equation}
Observing that $\int_0^te^{-2\lambda_{\alpha}^2(t-s)}\,ds
\leq\,\frac{1}{2}\,\lambda_{\alpha}^{-2}$ for $\alpha\in\nset^d$,
and that
\begin{equation*}
\begin{split}
\sup_{y,y'\in
D_{\mu}}\big|\varepsilon_{\alpha}(y)-\varepsilon_{\alpha}(y')\big|
\leq&\,2^{\frac{d}{2}+1}\,\min\left\{1,\tfrac{\pi}{2}\,d^{\frac{1}{2}}
\,\Delta{x}\,|\alpha|_{\ssy\nset^d}\right\}\\
\leq&\,2^{\frac{d}{2}+1-\gamma}\,\pi^{\gamma}
\,d^{\frac{\gamma}{2}} \,\Delta{x}^{\gamma}
\,|\alpha|_{\ssy\nset^d}^{\gamma}, \quad\forall\,\gamma\in[0,1], \
\ \forall\,\alpha\in\nset^d,
\ \ \forall\,\mu\in{\mathcal J}^d_{\star},\\
\end{split}
\end{equation*}
\eqref{corv2} yields
\begin{equation}\label{marine_poros1}
\Theta_{\ssy
A}^2(t)\leq\,2^{d+1-2\gamma}\,d^{\gamma}\,\pi^{2\gamma-4}
\,(\Delta{x})^{2\gamma}
\,\sum_{\alpha\in\nset^d}\tfrac{1}{|\alpha|_{\ssy\nset^d}^{2(2-\gamma)}}.
\end{equation}
The series in \eqref{marine_poros1} converges when $2(2-\gamma)>d$
or equivalently $\gamma<\frac{4-d}{2}$. Thus, combining
\eqref{marine_poros1} and \eqref{ASYM_1}, we, finally, conclude
that
\begin{equation}\label{corv5}
\Theta_{\ssy A}(t)\leq\,C\,\epsilon^{-\frac{1}{2}}
\,\Delta{x}^{\frac{4-d}{2}-\epsilon}
\quad\forall\,\epsilon\in\left(0,\tfrac{4-d}{2}\right].
\end{equation}
\par\medskip
\par\noindent
{\tt Estimation of $\Theta_{\ssy B}(t)$}:
For $t\in (0,T]$, let
${\widehat N}(t):= \min\big\{\,\ell\in\nset:\ken 1\leq \ell\leq
N_{\star} \ken\text{\rm and}\ken t\leq t_{\ell}\,\big\}$
and
\begin{equation*}
{\widehat T}_n(t):=T_n\cap (0,t)=\left\{
\begin{aligned}
&T_n,\quad\quad\quad\quad\text{\rm if}\ken n<{\widehat N}(t)\\
&(t_{\ssy {\widehat N}(t)-1},t),\hskip0.31truecm
\text{\rm if}\ken n={\widehat N}(t)\\
\end{aligned}
\right.,\quad n=1,\dots,{\widehat N}(t).
\end{equation*}
Now, we use \eqref{GreenKernel} and the $(\cdot,\cdot)_{\ssy
0,D}-$orthogonality of $(\varepsilon_{\alpha})_{\alpha\in\nset^d}$
as follows
\begin{equation*}
\begin{split}
\Theta^2_{\ssy B}(t)&=
\tfrac{(\Delta{x})^d}{(\Delta{t}\,(\Delta{x})^d)^2}
\sum_{n\in{\mathcal N}_{\star}} \sum_{\mu\in{\mathcal
J}_{\star}^d} \int_{\ssy D}\Bigg\{\int_{\ssy T_n} \Bigg[\int_{\ssy
S_{n,\mu}}\Big[{\mathcal X}_{(0,t)}(s)\,G(t-s;x,y') -{\mathcal
X}_{(0,t)}(s')\,G(t-s';x,y')\Big]\,ds'dy'
\Bigg]^2\,ds\Bigg\}\,dx\\
&\hskip-0.5truecm=
\tfrac{(\Delta{x})^d}{(\Delta{t}\,(\Delta{x})^d)^2}
\sum_{\alpha\in\nset^d}\left[\,\sum_{\mu\in{\mathcal J}_{\star}^d}
\Big(\int_{\ssy
D_{\mu}}\varepsilon_{\alpha}(y')\,dy'\Big)^2\,\right]
\left[\,\sum_{n=1}^{{\widehat N}(t)}\int_{\ssy T_n}
\Big(\,\int_{\ssy T_n}\Big(\,{\mathcal
X}_{(0,t)}(s)\,e^{-\lambda_{\alpha}^2(t-s)} -{\mathcal
X}_{(0,t)}(s')\,e^{-\lambda_{\alpha}^2(t-s')}\,\Big)\,ds'\,\Big)^2
\,ds\,\right]\\
\end{split}
\end{equation*}
which yields that
\begin{equation}\label{corv6}
\Theta^2_{\ssy B}(t)\leq\,2^d\,
\sum_{\alpha\in\nset^d}\left(\,\tfrac{1}{(\Delta{t})^2}\,
\sum_{n=1}^{{\widehat N}(t)}\Psi_n^{\alpha}(t)\,\right),
\end{equation}
where
\begin{equation*}
\Psi_n^{\alpha}(t):=\int_{\ssy T_n} \Big(\,\int_{\ssy T_n}\left(
{\mathcal X}_{(0,t)}(s)\,e^{-\lambda_{\alpha}^2(t-s)} -{\mathcal
X}_{(0,t)}(s')\,e^{-\lambda_{\alpha}^2(t-s')}
\right)\,ds'\,\Big)^2\,ds.
\end{equation*}
\par
Let $\alpha\in\nset^d$ and $n\in\{1,\dots,{\widehat N}(t)-1\}$.
Then, we have
\begin{equation*}
\begin{split}
\Psi_n^{\alpha}(t) &=\,\int_{\ssy T_n}\Big(\,\int_{\ssy T_n}
\!\!\int_s^{s'} \lambda_{\alpha}^2\,
e^{-\lambda_k^2(t-\tau)}\,d\tau ds'\,\Big)^2\,ds\\
&\leq\,\int_{\ssy T_n} \Big(\,\int_{\ssy
T_n}\!\!\int_{t_{n-1}}^{\max\{s',s\}}
\lambda_{\alpha}^2\,e^{-\lambda_{\alpha}^2(t-\tau)}
\,d\tau ds'\,\Big)^2\,ds\\
&\leq\,2\int_{\ssy T_n}\Big(\, \int_{\ssy
T_n}\!\!\int_{t_{n-1}}^{s'}
\lambda_{\alpha}^2\,e^{-\lambda_{\alpha}^2(t-\tau)}\,d\tau
ds'\,\Big)^2\,ds +2\int_{\ssy T_n}\Big(\, \int_{\ssy
T_n}\!\!\int_{t_{n-1}}^{s}\lambda_{\alpha}^2\,
e^{-\lambda_{\alpha}^2(t-\tau)}\,d\tau\,ds'\,\Big)^2\,ds\\
&\leq\,2\,\Delta{t}\,\Big(\, \int_{\ssy
T_n}\!\!\int_{t_{n-1}}^{s'} \lambda_{\alpha}^2\,
e^{-\lambda_{\alpha}^2(t-\tau)}\,d\tau ds'\Big)^2
+2\,(\Delta{t})^2\,\int_{\ssy T_n}\Big(\int_{t_{n-1}}^{s}
\lambda_{\alpha}^2\,
e^{-\lambda_{\alpha}^2(t-\tau)}\,d\tau\Big)^2\,ds,\\
\end{split}
\end{equation*}
from which, using the Cauchy-Schwarz inequality and
integrating by parts, we obtain
\begin{equation*}
\begin{split}
\Psi_n^{\alpha}(t)&\leq\,4\,(\Delta{t})^2\,\int_{\ssy T_n}
\Big(\,e^{-\lambda_{\alpha}^2(t-s)}
-e^{-\lambda_{\alpha}^2(t-t_{n-1})}\,\Big)^2\,ds\\
&\leq\,4\,(\Delta{t})^2\,\big(1-e^{-\lambda_{\alpha}^2\Delta{t}}\big)^2
\int_{\ssy T_n}e^{-2\lambda_{\alpha}^2(t-s)}\,ds\\
&\leq\,2\,(\Delta{t})^2\,\big(1-
e^{-\lambda_{\alpha}^2\Delta{t}}\big)^2
\,\tfrac{e^{-\lambda_{\alpha}^2(t-t_n)}
-e^{-\lambda_{\alpha}^2(t-t_{n-1})}}{\lambda_{\alpha}^2}\cdot\\
\end{split}
\end{equation*}
Thus, by summing with respect to $n$, we obtain
\begin{equation}\label{corv7}
\tfrac{1}{(\Delta{t})^2}\,\sum_{n=1}^{{\widehat
N}(t)-1}\Psi_n^{\alpha}(t)\leq\,2\,
\tfrac{(1-e^{-\lambda_{\alpha}^2\Delta{t}})^2}{\lambda_{\alpha}^2}\cdot
\end{equation}
Considering, now, the case $n={\widehat N}(t)$, we have
\begin{equation}\label{gaga}
\Psi_{\ssy {\widehat N}(t)}^{\alpha}(t)=\Psi^{\alpha}_{\ssy
A}(t)+\Psi_{\ssy B}^{\alpha}(t)
\end{equation}
with
\begin{equation*}
\begin{split}
\Psi_{\ssy A}^{\alpha}(t)&:=\int_{t_{{\widehat N}(t)-1}}^t
\left(\,\int_{t_{{\widehat N}(t)-1}}^t
\int_{s'}^s\lambda_{\alpha}^2e^{-\lambda_{\alpha}^2(t-\tau)}\,d\tau{ds'}
+\int_t^{t_{{\widehat N}(t)}} e^{-\lambda_{\alpha}^2(t-s)}\,ds'
\,\right)^2\,ds\\
\Psi_{\ssy B}^{\alpha}(t)&:=\int_t^{t_{{\widehat N}(t)}}\left(\,
\int_{t_{{\widehat N}(t)}-1}^t
e^{-\lambda_{\alpha}^2(t-s')}\,ds'\,\right)^2\,ds.\\
\end{split}
\end{equation*}
Then, we have
\begin{equation*}
\begin{split}
\Psi_{\ssy
B}^{\alpha}(t)&\leq\,\tfrac{\Delta{t}}{\lambda_{\alpha}^4}\,
\Big[\,1-e^{-\lambda_{\alpha}^2\,
\big(\,t-t_{{\widehat N}(t)-1}\,\big)}\,\Big]^2\\
&\leq\,\tfrac{\Delta{t}}{\lambda_{\alpha}^4}\,
\big(\,1-e^{-\lambda_{\alpha}^2\,\Delta{t}}\,)^2\\
\end{split}
\end{equation*}
and
\begin{equation*}
\begin{split}
\Psi_{\ssy A}^{\alpha}(t)&\leq\,\int_{t_{{\widehat N}(t)-1}}^t
\left[ \int_{t_{{\widehat N}(t)-1}}^t
\int_{s'}^s\lambda_{\alpha}^2e^{-\lambda_{\alpha}^2(t-\tau)}\,d\tau{ds'}
+\Delta{t}\,\,\,e^{-\lambda_{\alpha}^2(t-s)}
\right]^2\,ds \\
&\leq\,2\,\int_{t_{{\widehat N}(t)-1}}^t \left[ \int_{t_{{\widehat
N}(t)-1}}^t\int_{s'}^s
\lambda_{\alpha}^2e^{-\lambda_{\alpha}^2(t-\tau)}\,d\tau{ds'}
\right]^2\,ds +\tfrac{(\Delta{t})^2}{\lambda_{\alpha}^2}\,
\left[\,1-e^{-2\lambda_{\alpha}^2\left(\,t-t_{{\widehat N}(t)-1}\,\right)}\,\right]\\
&\leq\, 2\,\int_{t_{{\widehat N}(t)-1}}^t \left[
\int_{t_{{\widehat N}(t)-1}}^t\int_{t_{{\widehat
N}(t)-1}}^{\max\{s,s'\}}
\lambda_{\alpha}^2e^{-\lambda_{\alpha}^2(t-\tau)}\,d\tau{ds'}
\right]^2\,ds
+\tfrac{(\Delta{t})^2}{\lambda_{\alpha}^2}
\,\big(\,1-e^{-2\lambda_{\alpha}^2\,\Delta{t}}\,\big)\\
&\leq \,8\,(\Delta{t})^2 \int_{t_{{\widehat N}(t)-1}}^t
\left[\,\int_{t_{{\widehat N}(t)-1}}^s\lambda_{\alpha}^2
e^{-\lambda_{\alpha}^2(t-\tau)}\,d\tau\,\right]^2\,ds
+\tfrac{(\Delta{t})^2}{\lambda_{\alpha}^2}
\,\big(\,1-e^{-2\lambda_{\alpha}^2\,\Delta{t}}\,\big)\\
&\leq \,8\,(\Delta{t})^2\, \int_{t_{{\widehat N}(t)-1}}^t
\left[\,e^{-\lambda_{\alpha}^2(t-s)}
-e^{-\lambda_{\alpha}^2(t-t_{{\widehat N}(t)-1})}\,\right]^2\,ds
+\tfrac{(\Delta{t})^2}{\lambda_{\alpha}^2}
\,\big(\,1-e^{-2\lambda_{\alpha}^2\,\Delta{t}}\,\big),\\
\end{split}
\end{equation*}
which, along with \eqref{gaga}, gives
\begin{equation*}
\Psi_{\ssy {\widehat
N}(t)}^{\alpha}\leq\,5\tfrac{(\Delta{t})^2}{\lambda_{\alpha}^2}
\,\big(\,1-e^{-2\lambda_{\alpha}^2\,\Delta{t}}\,\big)
+\tfrac{\Delta{t}}{\lambda_{\alpha}^4}\,
\big(\,1-e^{-\lambda_{\alpha}^2\Delta{t}}\,\big)^2\,\cdot
\end{equation*}
Since the mean value theorem yields:
$1-e^{-\lambda_{\alpha}^2\Delta{t}}\leq\,\lambda_{\alpha}^2\,\Delta{t}$,
the above inequality takes the form
\begin{equation}\label{corv8}
\tfrac{1}{(\Delta{t})^2}\,\Psi_{\ssy {\widehat N}(t)}^{\alpha}\leq
\,6\,\tfrac{1-e^{-2\lambda_{\alpha}^2\,\Delta{t}}}{\lambda_{\alpha}^2}\,\cdot
\end{equation}
\par
Combining \eqref{corv6}, \eqref{corv7} and \eqref{corv8} we obtain
\begin{equation}\label{corv9}
\Theta_{\ssy B}^2(t)\leq \,8\,\sum_{\alpha\in\nset^d}
\tfrac{1-e^{-2\lambda_{\alpha}^2\,\Delta{t}}}{\lambda_{\alpha}^2}
\,\cdot
\end{equation}
Now, combine \eqref{corv9} and \eqref{ASYM_2} to arrive at
\begin{equation}\label{corv12}
\Theta_{\ssy B}(t)\leq
\,C\,(p_d(\Delta{t}^{\frac{1}{4}}))^{\frac{1}{2}}
\,\,\Delta{t}^{\frac{4-d}{8}}.
\end{equation}
\par
The error bound \eqref{ModelError} follows by observing that
$\Theta(0)=0$ and combining the bounds \eqref{corv1},
\eqref{corv5} and \eqref{corv12}.
\end{proof}
%
%
%
%
%
%
\section{Time-Discrete Approximations}\label{SECTION3}
\par
The Backward Euler time-discrete approximations to the solution
${\widehat u}(\tau_m,\cdot)$ of the problem \eqref{AC2} are
defined as follows:
first, set
\begin{equation}\label{BackE1}
{\widehat U}^0:=0,
\end{equation}
and then, for $m=1,\dots,M$, find ${\widehat U}^m\in {\bfdot
H}^4(D)$ such that
\begin{equation}\label{BackE2}
{\widehat U}^m-{\widehat U}^{m-1} +k_m\,\Delta^2{\widehat U}^m
=\int_{\ssy\Delta_m}{\widehat W}\,ds\quad\text{\rm a.s.}.
\end{equation}
%
%
\par
To develop an error estimate in a discrete in time
$L^{\infty}_t(L^2_{\ssy P}(L_x^2))$ norm for the above
time-discrete approximations, we need an error estimate for the
Backward Euler time-discrete approximations, $(W^m)_{m=0}^{\ssy
M}$, of the solution $w$ to the deterministic problem
\eqref{Det_Parab}, given below:
First, set
\begin{equation}\label{BEDet1}
W^0:=w_0.
\end{equation}
Then, for $m=1,\dots,M$, find $W^m\in{\bfdot H}^4(D)$ such that
\begin{equation}\label{BEDet2}
W^m-W^{m-1} +k_m\,\Delta^2W^m=0.
\end{equation}
%
%
%
\begin{proposition}\label{DetPropo1}
Let $(W^m)_{m=0}^{\ssy M}$ be the Backward Euler time-discrete
approximations of the solution $w$ of the problem
\eqref{Det_Parab} defined in \eqref{BEDet1}--\eqref{BEDet2}. If
$w_0\in{\bfdot H}^2(D)$, then, there exists a constant $C>0$,
independent of $T$, $\Delta{t}$, $\Delta{x}$, $M$ and
$(k_m)_{m=1}^{\ssy M}$, such that
\begin{equation}\label{Ydaspis900}
\Bigg(\,\sum_{m=1}^{\ssy M}k_m\, \|W^m-w(\tau_m,\cdot)\|_{\ssy
0,D}^2 \,\Bigg)^{\frac{1}{2}} \leq\,C\,(k_{\ssy\rm max})^{\theta}
\,\|w_0\|_{\ssy{\bfdot H}^{4\theta-2}}
\quad\forall\,\theta\in[0,1].
\end{equation}
\end{proposition}
%
%
%
%
%
%
%
%
\begin{proof}
The proof is omitted since it is moving along the lines of the
proof of the one dimensional case which is exposed in
Proposition~5.1 of \cite{KZ2008}.
\end{proof}
%
%
%
%
\par
Next theorem proves a discrete in time $L^{\infty}_t(L^2_{\ssy
P}(L^2_x))$ convergence estimate for the Backward Euler time
discrete approximations of ${\widehat u}$, over a uniform
partition of $[0,T]$.
%
%
%
%
\begin{theorem}\label{TimeDiscreteErr1}
Let ${\widehat u}$ be the solution of \eqref{AC2} and $\{{\widehat
U}^m\}_{m=0}^{\ssy M}$ be the Backward Euler time-discrete
approximations specified in \eqref{BackE1}--\eqref{BackE2}. If
$k_m=\Delta{\tau}$ for $m=1,\dots,M$, then there exists constant
$C>0$, independent of $T$, $\Delta{t}$, $\Delta{x}$ and
$\Delta\tau$, such that
\begin{equation}\label{ElPasso}
\max_{1\leq m \leq {\ssy M}} \left\{{\mathbb E}\left[ \|{\widehat
U}^m-{\widehat u}(\tau_m,\cdot)\|_{\ssy
0,D}^2\right]\right\}^{\half} \leq \,C
\,\,\,{\widetilde\omega}(\Delta\tau,\epsilon)
\,\,\,\Delta{\tau}^{\frac{4-d}{8}-\epsilon},
\quad\forall\,\epsilon\in(0,\tfrac{4-d}{8}],
\end{equation}
where
${\widetilde\omega}(\Delta\tau,\epsilon):=[\epsilon^{-\frac{1}{2}}
+(\Delta\tau)^{\epsilon}
\,(p_d(\Delta\tau^{\frac{1}{4}}))^{\frac{1}{2}}]$.
\end{theorem}
%
%
%
%
%
%
%
%
%
\begin{proof}
Let $I:L^2(D)\to L^2(D)$ be the identity operator,
$\Lambda:L^2(D)\to{\bfdot H}^4(D)$ be the inverse elliptic
operator $\Lambda:=(I+\Delta{\tau}\,\Delta^2)^{-1}$ which has
Green function $G_{\ssy\Lambda}(x,y)=
\sum_{\alpha\in\nset^d}\frac{\varepsilon_{\alpha}(x)
\,\varepsilon_{\alpha}(y)}{1+\Delta\tau\lambda_{\alpha}^2}$, i.e.
$\Lambda{f}(x)=\int_{\ssy D}G_{\ssy\Lambda}(x,y)f(y)\,dy$ for
$x\in{\overline D}$ and $f\in L^2(D)$. Obviously,
$G_{\ssy\Lambda}(x,y)=G_{\ssy\Lambda}(y,x)$ for $x,y\in D$, and
$G\in L^2(D\times D)$.
Also, for $m\in\nset$, we
denote by $G_{{\ssy\Lambda},m}$ the Green function of
$\Lambda^m$.
Thus, from \eqref{BackE2}, using an induction argument, we
conclude that
${\widehat U}^m=\sum_{j=1}^{\ssy m} \int_{\ssy\Delta_j}
\Lambda^{m-j+1}{\widehat W}(\tau,\cdot)\,d\tau$
for $m=1,\dots,M$, which is written,
equivalently, as follows:
\begin{equation}\label{Anaparastash1}
{\widehat U}^m(x)
=\int_0^{\tau_m}\!\!\!\int_{\ssy D}
\,{\widehat {\mathcal K}}_m(\tau;x,y)\,{\widehat W}(\tau,y)\,dyd\tau
\quad\forall\,x\in{\overline D}, \ken m=1,\dots,M,
\end{equation}
where
${\widehat {\mathcal K}}_m(\tau;x,y):=\sum_{j=1}^m{\mathcal
X}_{\ssy\Delta_j}(\tau) \,G_{{\ssy \Lambda},m-j+1}(x,y)
\quad\forall\,\tau\in[0,T],\ken\forall\,x,y\in D$.
\par
Let $m\in\{1,\dots,M\}$ and ${\mathcal E}^m:={\mathbb E}\big[
\|{\widehat U}^m -{\widehat u}(\tau_m,\cdot)\|_{\ssy 0,D}^2\big]$.
First, we use \eqref{Anaparastash1}, \eqref{HatUform},
\eqref{Ito_Isom} and \eqref{HSxar}, to obtain
\begin{equation*}
\begin{split}
{\mathcal E}^m&={\mathbb E}\Big[\,\int_{\ssy D}\Big( \int_0^{\ssy
T} \!\!\!\int_{\ssy D} {\mathcal X}_{(0,\tau_m)}(\tau)
\,\big[{\widehat{\mathcal
K}}_m(\tau;x,y)-G(\tau_m-\tau;x,y)\big]\,
{\widehat W}(\tau,y)\,dyd\tau\Big)^2\,dx\Big]\\
%
%
&=\tfrac{1}{\Delta{t}\,(\Delta{x})^d} \int_{\ssy D}
\Bigg\{\sum_{n\in{\mathcal N}_{\star}}\sum_{\mu\in{\mathcal
J}_{\star}^d} \Bigg(\int_{\ssy S_{n,\mu}} {\mathcal
X}_{(0,\tau_m)}(\tau) \,\big[{\widehat{\mathcal K}}_m(\tau;x,y)
-G(\tau_m-\tau;x,y)\big]\,d\tau dy\Bigg)^2\Bigg\}\,dx\\
\end{split}
\end{equation*}
Then, we apply the Cauchy-Schwarz inequality and \eqref{HSxar} to
arrive at
\begin{equation*}
\begin{split}
%
{\mathcal E}^m&\leq\int_0^{\tau_m}\left(\int_{\ssy D}\!\int_{\ssy
D} \,\big[{\widehat{\mathcal K}}_m(\tau;x,y)
-G(\tau_m-\tau;x,y)\big]^2\,dydx\right)\,d\tau\\
%
%
&\leq\sum_{\ell=1}^{m}\int_{\ssy\Delta_{\ell}} \left(\int_{\ssy
D}\!\int_{\ssy D} \,\big[G_{{\ssy\Lambda},m-\ell+1}(x,y)
-G(\tau_m-\tau;x,y)\big]^2\,dydx\right)\,d\tau.\\
%
%
&\leq\sum_{\ell=1}^m\int_{\ssy\Delta_{\ell}} \|\Lambda^{m-\ell+1}
-{\mathcal S}(\tau_m-\tau)\|_{\ssy\rm
HS}^2\,d\tau.\\
\end{split}
\end{equation*}
Now, we introduce the splitting
\begin{equation}\label{mainerrorF}
{\mathcal E}^m\leq\,{\mathcal B}_1^m+{\mathcal B}_2^m,
\end{equation}
where
\begin{equation*}
\begin{split}
{\mathcal B}_1^m&:=\,2\,\sum_{\ell=1}^m\int_{\ssy\Delta_{\ell}}
\,\|\Lambda^{m-\ell+1}-{\mathcal S}(\tau_m-\tau_{\ell-1})
\|_{\ssy\rm HS}^2\,d\tau,\\
{\mathcal B}_2^m&:=\,2\,\sum_{\ell=1}^m\int_{\ssy\Delta_{\ell}}\,
\|{\mathcal S}(\tau_m-\tau_{\ell-1}) -{\mathcal
S}(\tau_m-\tau)\|_{\ssy\rm HS}^2
\,d\tau.\\
\end{split}
\end{equation*}
\par\noindent
{\tt Estimation of ${\mathcal B}_1^m$}: By the definition of the
Hilbert-Schmidt norm, we have
\begin{equation*}
\begin{split}
%
%
{\mathcal B}_1^m&\leq\,2\,\Delta\tau\,\sum_{\ell=1}^m
\left(\,\sum_{\alpha\in\nset^d}
\|\Lambda^{m-\ell+1}\varepsilon_{\alpha} -{\mathcal
S}(\tau_m-\tau_{\ell-1})\varepsilon_{\alpha}
\|^2_{\ssy 0,D}\,\right)\\
%
%
&\leq\,2\,\sum_{\alpha\in\nset^d}\left(\,
\sum_{\ell=1}^m\,\Delta\tau\, \|\Lambda^{m-\ell+1}\varepsilon_k
-{\mathcal S}(\tau_m-\tau_{\ell-1})\varepsilon_k\|^2_{\ssy 0,D}
\,\right)\\
%
%
&\leq\,2\,\sum_{\alpha\in\nset^d}\left(\, \sum_{\ell=1}^m
\,\Delta\tau\,\|\Lambda^{\ell}\varepsilon_k -{\mathcal
S}(\tau_{\ell})\varepsilon_k
\|^2_{\ssy 0,D}\,\right).\\
\end{split}
\end{equation*}
Let $\theta\in[0,\frac{4-d}{8})$. Using the deterministic error
estimate \eqref{Ydaspis900}, we obtain
\begin{equation}\label{trabajito}
\begin{split}
{\mathcal B}_1^m&\leq\,C\,\Delta{\tau}^{2\theta}
\,\sum_{\alpha\in\nset^d}\|\varepsilon_{\alpha}\|^2_{\ssy
{\bfdot H}^{4\theta-2}}\\
&\leq\,C\,\Delta{\tau}^{2\theta}\, \,\sum_{\alpha\in\nset^d}
\tfrac{1}{|\alpha|_{\ssy \nset^d}^{4(1-2\theta)}}.\\
\end{split}
\end{equation}
The convergence of the series is ensured because $4(1-2\theta)>d$.
\par\noindent
{\tt Estimation of ${\mathcal B}_2^m$}: Using, again,
the definition of the
Hilbert-Schmidt norm we have
\begin{equation}\label{Ydaspis950}
{\mathcal B}_2^m=\,2\,\sum_{\alpha\in\nset^d}\left(\,
\sum_{\ell=1}^m\int_{\ssy\Delta_{\ell}} \|{\mathcal
S}(\tau_m-\tau_{\ell-1})\varepsilon_{\alpha} -{\mathcal
S}(\tau_m-\tau)\varepsilon_{\alpha} \|^2_{\ssy 0,D}
\,d\tau\,\right).
\end{equation}
Observing that
${\mathcal
S}(t)\varepsilon_{\alpha}=e^{-\lambda_{\alpha}^2t}\,\varepsilon_{\alpha}$
for $t\ge 0$, \eqref{Ydaspis950} yields
\begin{equation*}
\begin{split}
{\mathcal B}_2^m&=\,2\,\sum_{\alpha\in\nset^d}\left[\,
\sum_{\ell=1}^m\int_{\ssy\Delta_{\ell}} \left(\int_{\ssy D}\left[
e^{-\lambda_{\alpha}^2(\tau_m-\tau_{\ell-1})}
-e^{-\lambda_{\alpha}^2(\tau_m-\tau)}\right]^2
\varepsilon_{\alpha}^2(x)\,dx\right)
\,d\tau\,\right]\\
&=\,2\,\sum_{\alpha\in\nset^d}\left[\,
\sum_{\ell=1}^m\int_{\ssy\Delta_{\ell}}
e^{-2\lambda_{\alpha}^2(\tau_m-\tau)} \left[1-
e^{-\lambda_{\alpha}^2(\tau-\tau_{\ell-1})}\right]^2
\,d\tau\,\right]\\
&\leq\,2\,\sum_{\alpha\in\nset^d}\big(1-e^{-\lambda_{\alpha}^2\,\Delta\tau
}\,\big)^2 \left[\, \int_0^{\tau_m}
e^{-2\lambda_{\alpha}^2(\tau_m-\tau)}
\,d\tau\,\right]\\
&\leq\,\sum_{\alpha\in\nset^d}
\tfrac{1-e^{-2\lambda_{\alpha}^2\,\Delta\tau}}{\lambda_{\alpha}^2},
\end{split}
\end{equation*}
from which, applying \eqref{ASYM_2}, we obtain
\begin{equation}\label{Ydaspis952}
{\mathcal B}_2^m\leq\,C\,\,p_d(\Delta\tau^{\frac{1}{4}})
\,\,\Delta\tau^{\frac{4-d}{4}}.
\end{equation}
\par
Thus, we obtain the estimate \eqref{ElPasso} as a conclusion of
\eqref{mainerrorF}, \eqref{trabajito}, \eqref{Ydaspis952} and
\eqref{ASYM_1}.
\end{proof}
%
%
%
%
%
\section{Space-Discrete Approximations}\label{SECTION4}
%
%
%
Let $r\in\{2,3,4\}$. The space-discrete approximation of the
solution ${\widehat u}$ of \eqref{AC2} is a stochastic function
${\widehat u}_h:[0,T]\to M_h$ such that
\begin{equation}\label{Semi}
\aligned
\partial_t\widehat{u}_h+B_h&\widehat{u}_h
=P_h{\widehat W}
\quad\text{\rm on}\ken(0,T],\\
&{\widehat u}_h(0)=0\\
\endaligned\quad\quad\mbox{\rm a.s.}.
\end{equation}
To develop an $L^{\infty}_t(L^2_{\ssy P}(L_x^2))$ convergence
estimate for the space-discrete approximation ${\widehat u}_h$, we
will derive first an $L^2_t(L^2_x)$ error estimate for the
corresponding space-discrete approximation $w_h$ of the solution
$w$ of \eqref{Det_Parab} (cf. \cite{YubinY04} and \cite{BinLi}),
which is a function $w_h:[0,T]\to M_h$ such that
\begin{equation}\label{SemiHomo}
\begin{gathered}
\partial_tw_h+B_hw_h=0\quad\text{\rm on}\ken(0,T],\\
w_h(0)=P_hw_0.\\
\end{gathered}
\end{equation}
Since $w_h$ can be considered as the value of a linear operator of
the initial condition $w_0$, we will write it as
$w_h(t,\cdot)=[{\mathcal S}_h(t)w_0](\cdot)$ for $t\in[0,T]$.
Thus, by Duhamel's principle (cf. \cite{Thomee}), we have
\begin{equation}\label{Duhamel_uhat}
{\widehat u}_h(t,x)=\int_0^t [{\mathcal S}_h(t-s) {\widehat
W}(s,\cdot)](x)\,ds\quad\text{\rm a.s.}.
\end{equation}
%
%
%
\begin{proposition}\label{LowRegSD}
Let $r\in\{2,3,4\}$, $w$ be the solution of \eqref{Det_Parab} and
$w_h$ be its space-discrete approximation given in
\eqref{SemiHomo}. If $w_0\in{\bfdot H}^3(D)$, then, there exists a
constant $C>0$, independent of $T$ and $h$, such that
\begin{equation}\label{karx101}
\Bigg(\int_0^{\ssy T}\|w-w_h\|^2_{\ssy 0,D}\,dt\Bigg)^{\half}
\leq\,C\,h^{{\widetilde\nu}(r,\theta)} \,\|w_0\|_{\ssy {\bfdot
H}^{{\widetilde\xi}(r,\theta)}} \quad\forall\,\theta\in[0,1],
\end{equation}
where
\begin{equation}\label{Basilico_4}
{\widetilde\nu}(r,\theta):=\left\{ \aligned
&2\,\theta\quad\hskip0.25truecm\text{\rm if}\ken r=2\\
&4\,\theta\quad\ken\text{\rm if}\ken r=3\\
&5\,\theta\quad\ken\text{\rm if}\ken r=4\\
\endaligned
\right.
\quad\quad\text{\rm and}\quad\quad
{\widetilde\xi}(r,\theta):=\left\{ \aligned
&3\theta-2 \quad\text{\rm if}\ken r=2\\
&4\theta-2\quad\text{\rm if}\ken r=3\\
&5\theta-2\quad\text{\rm if}\ken r=4\\
\endaligned
\right..
\end{equation}
\end{proposition}
%
%
%
%
\begin{proof}
Let $e:=w-w_h$ and $\rho:=(T_{\ssy B,h}-T_{\ssy B})\Delta^2w$. We
will derive \eqref{karx101} by interpolation, after showing that
it holds for $\theta=1$ and $\theta=0$.
\par
Observing that $T_{{\ssy B},h}e_t+e=\rho$ on $[0,T]$, and then
taking the $(\cdot,\cdot)_{\ssy 0,D}$ inner product with $e$, we
easily arrive at
\begin{equation}\label{baga1}
\int_0^{\ssy T}\|e\|_{\ssy 0,D}^2\,dt \leq\,\int_0^{\ssy
T}\|\rho\|_{\ssy 0,D}^2\,dt.
\end{equation}
For $r=2$, using \eqref{baga1}, \eqref{ARA2}, \eqref{H_equiv} and
\eqref{Reggo3}, we have
\begin{equation}\label{baga2}
\begin{split}
\left(\int_0^{\ssy T}\|e\|^2_{\ssy 0,D}\,dt\right)^{\frac{1}{2}}
&\leq\,C\,h^2\,\left(\int_0^{\ssy T}
\|w\|_{\ssy{\bfdot H}^3}^2\,dt\right)^{\frac{1}{2}}\\
&\leq\,C\,h^2\,\|w_0\|_{\ssy{\bfdot H}^1}.\\
\end{split}
\end{equation}
Also, for $r=3,4$, combining \eqref{baga1}, \eqref{ARA2},
\eqref{H_equiv} and \eqref{Reggo3} we get
\begin{equation}\label{baga3}
\begin{split}
\left(\int_0^{\ssy T}\|e\|^2_{\ssy 0,D}
\,dt\right)^{\frac{1}{2}}&\leq\,C\,h^{r+1}\,\left(\int_0^{\ssy T}
\|w\|_{\ssy {\bfdot H}^{r+1}}^2\,dt\right)^{\frac{1}{2}}\\
&\leq\,C\,h^{r+1}\,\|w_0\|_{\ssy {\bfdot H}^{r-1}}.\\
\end{split}
\end{equation}
Thus, relations \eqref{baga2} and \eqref{baga3} yield
\eqref{karx101} for $\theta=1$.
\par
Since $T_{\ssy B}w_t+w=0$ on $[0,T]$, we obtain $(T_{\ssy
B}w_t,w)_{\ssy 0,D}+\|w\|_{\ssy 0,D}^2=0$ on $[0,T]$, which,
along with \eqref{TB-prop1}, yields $\tfrac{d}{dt}\|T_{\ssy
E}w\|_{\ssy 0,D}^2+\|w\|_{\ssy 0,D}^2=0$ on $[0,T]$.
Then, integrating over $[0,T]$ and using \eqref{ElReg1}, we get
\begin{equation}\label{karx103}
\left(\int_0^{\ssy T}\|w\|_{\ssy
0,D}^2\,dt\right)^{\frac{1}{2}}\leq\,C\,\|w_0\|_{\ssy -2, D}.
\end{equation}
Since $T_{\ssy B,h}\partial_tw_h+w_h=0$ on $[0,T]$, we obtain
$(T_{\ssy B,h}\partial_tw_h,w_h)_{\ssy 0,D}+\|w_h\|_{\ssy
0,D}^2=0$ on $[0,T]$, which, along with \eqref{adjo_Bh}, yields
$\tfrac{d}{dt}(T_{\ssy B,h}w_h,w_h)_{\ssy 0,D}+\|w_h\|_{\ssy
0,D}^2=0$ on $[0,T]$.
Then, integrating over $[0,T]$ and using \eqref{PanwFragma1}, we
have
\begin{equation}\label{karx104}
\begin{split}
\left(\int_0^{\ssy T}\|w_h\|_{\ssy 0,D}^2\,dt\right)^{\frac{1}{2}}
&\leq \|\Delta T_{\ssy B,h}P_hw_0\|_{\ssy 0,D}\\
&\leq \|\Delta T_{\ssy B,h}w_0\|_{\ssy 0,D}\\
&\leq\,C\,\|w_0\|_{\ssy -2, D}.\\
\end{split}
\end{equation}
Hence, from \eqref{karx103}, \eqref{karx104} and
\eqref{minus_equiv}, we obtain
$\left(\int_0^{\ssy T}\|e\|_{\ssy 0,D}^2\,dt\right)^{\frac{1}{2}}
\leq\,C\,\|w_0\|_{\ssy{\bfdot H}^{-2}}$,
which yields \eqref{karx101} with $\theta=0$.
\end{proof}
%
%
%
%
\par
Next lemma shows that a discrete analogue of \eqref{deter_green} holds.
%
%
%
\begin{lemma}\label{prasino}
Let $r\in\{2,3,4\}$ and $w_h$ be the space-discrete approximation
of the solution $w$ of \eqref{Det_Parab} defined in
\eqref{SemiHomo}. Then, there exists a map $G_h:[0,T]\rightarrow
C(\overline{D\times D})$ such that
\begin{equation}\label{semi_green}
w_h(t;x)=\int_{\ssy D} G_h(t;x,y)\,w_0(y)\,dy
\quad\forall\,t\in[0,T],\ken\forall\,x\in{\overline D},
\end{equation}
and $G_h(t;x,y)=G_h(t;y,x)$ for $x,y\in\overline{D}$ and $t\in[0,T]$.
\end{lemma}
%
%
%
%
%
\begin{proof}
Let $\text{\rm dim}(M_h)=n_h$ and $\gamma_h:M_h\times
M_h\rightarrow\rset$ be an inner product on $M_h$ given by
$\gamma_h(\chi_{\ssy A},\chi_{\ssy B}):=(\Delta\chi_{\ssy A},
\Delta\chi_{\ssy B})_{\ssy 0,D}$ for $\chi_{\ssy A}$, $\chi_{\ssy
B}\in M_h$.
We can construct a basis $(\chi_j)_{j=1}^{n_h}$ of $M_h$ which is
$L^2(D)-$orthonormal, i.e., $(\chi_i,\chi_j)_{\ssy
0,D}=\delta_{ij}$ for $i,j=1,\dots,n_h$, and
$\gamma_h-$orthogonal, i.e., there are
$(\lambda_{h,\ell})_{\ell=1}^{\ssy n_h}\subset(0,+\infty)$ such
that $\gamma_h(\chi_i,\chi_j)=\lambda_{h,i}\,\delta_{ij}$ for
$i,j=1,\dots,n_h$ (see Section 8.7 in \cite{Golub}).
Thus, there exists a map $\omega:[0,T]\rightarrow{\rset}^{n_h}$
such that $w_h(t;x)=\sum_{j=1}^{n_h}\omega_j(t)\,\chi_j(x)$. Since
$w_h(0)=P_hw_0$, it follows that $\omega_j(0)=(w_0,\chi_j)_{\ssy
0,D}$ for $j=1,\dots,n_h$. Now, \eqref{SemiHomo} yields that
$\frac{d}{dt}\omega(t)=B\,\omega(t)$ for $t\in[0,T]$, where
$B\in\rset^{n_h\times n_h}$ with
$B_{ij}:=-\gamma_h(\chi_i,\chi_j)=-\lambda_{h,i}\,\delta_{ij}$ for
$i,j=1,\dots,n_h$. Hence, it follows that
$\omega_{\ell}(t)=e^{-\lambda_{h,\ell}\,t}\,(w_0,\chi_{\ell})_{\ssy
0,D}$ for $t\in[0,T]$ and $\ell=1,\dots,n_h$,
which yields \eqref{semi_green} with $G_h(t;x,y)= \sum_{j=1}^{n_h}
e^{-\lambda_{h,j}\,t}\chi_j(x)\chi_j(y)$.
\end{proof}
%
%
\par
We are ready to derive a convergence estimate, in an
$L^{\infty}_t(L^2_{\ssy P}(L^2_x))$ norm, for the space-discrete
approximation ${\widehat u}_h$ to the solution ${\widehat u}$ of
the regularized problem.
%
%
\begin{theorem}\label{coconut12}
Let $r\in\{2,3,4\}$, $\widehat{u}$ be the solution of \eqref{AC2}
and ${\widehat u}_h$ be its space-discrete approximation defined
in \eqref{Semi}. Then, there exist a constant $C>0$, independent
of $T$, $\Delta{t}$, $\Delta{x}$ and $h$, such that
\begin{equation}\label{soho3}
\max_{[0,T]}\left\{{\mathbb E}\left[ \|{\widehat u}_h-{\widehat
u}\|^2_{\ssy 0,D}\right] \right\}^{\half}
\leq\,C\,\epsilon^{-\frac{1}{2}}\,h^{\nu(r,d)},
\quad\forall\,\epsilon\in(0,\nu(r,d)],
\end{equation}
where
\begin{equation}\label{soho4}
\nu(r,d):=\left\{ \aligned
&\tfrac{4-d}{3}\quad\text{\rm if}\ken r=2\\
&\tfrac{4-d}{2}\quad\text{\rm if}\ken r=3,4\\
\endaligned
\right..
\end{equation}
\end{theorem}
%
%
%
%
%
%
%
%
%
%
\begin{proof}
Let ${\widehat e}:={\widehat u}_h-{\widehat u}$ and $t\in(0,T]$.
Then, \eqref{Duhamel_uhat}, \eqref{semi_green}
and \eqref{HatUform} yield
\begin{equation*}
{\widehat e}(t,x)=\int_0^t\!\!\int_{\ssy D}
\big[\,G_h(t-s;x,y)-G(t-s;x,y)\,\big]\,{\widehat W}(s,y)\,dsdy
\quad\forall\,x\in\overline{D},
\quad\text{\rm a.s.}.
\end{equation*}
Thus, using the It{\^o} isometry
property of the stochastic integral and the Cauchy-Schwarz
inequality, we have
\begin{equation*}
\begin{split}
{\mathbb E}\left[\|e(t,\cdot)\|^2_{\ssy 0,D}\right] &={\mathbb
E}\Bigg[\int_{\ssy D}\left( \int_0^{\ssy T}\int_{\ssy D}{\mathcal
X}_{(0,t)}(s) \,\big[G_h(t-s;x,y)-G(t-s;x,y)\big]
\,{\widehat W}(s,y)\,dsdy\right)^2\,dx\Bigg]\\
&=\tfrac{1}{\Delta{t}\,(\Delta{x})^d} \int_{\ssy D}
\sum_{n\in{\mathcal N}_{\star}}\sum_{\mu\in{\mathcal
J}_{\star}^d}\left(\int_{\ssy S_{n,\mu}}{\mathcal X}_{(0,t)}(s')\,
\big[G_h(t-s';x,y')-G(t-s';x,y')\big]
\,ds'dy'\right)^2\,dx\\
&\leq\int_0^t \left(\int_{\ssy D}\!\int_{\ssy D}
\big[G_h(t-s;x,y)-G(t-s;x,y)\big]^2\,dydx\right)\,ds\\
\end{split}
\end{equation*}
which, along with \eqref{HSxar}, yields
\begin{equation}\label{soho1}
{\mathbb E}\left[\|e(t,\cdot)\|^2_{\ssy 0,D}\right] \leq\int_0^t
\big\|\,{\mathcal S}(s) -{\mathcal S}_h(s)\,\big\|^2_{\ssy\rm
HS}\,ds.
\end{equation}
Since $e(0,\cdot)=0$, we use \eqref{soho1},
the definition of the
Hilbert-Schmidt norm
and \eqref{karx101}, to obtain
\begin{equation}\label{soho2}
\begin{split}
\max_{[0,T]}{\mathbb E}\left[\|e\|^2_{\ssy 0,D}\right]
&\leq\,\int_0^{\ssy T}\Bigg(\,\sum_{\alpha\in\nset^d}\|\,{\mathcal
S}(s)\varepsilon_{\alpha}-{\mathcal S}_h(s)
\varepsilon_{\alpha}\|^2_{\ssy 0,D}\,\Bigg)\,ds\\
&\leq\,\sum_{\alpha\in\nset^d}\,\left(\int_0^{\ssy T}
\big\|{\mathcal S}(s)\varepsilon_k
-{\mathcal S}_h(s)\varepsilon_k\big\|^2_{\ssy 0,D}\,ds\,\right)\\
&\leq\,C\,h^{2{\widetilde\nu}(r,\theta)} \,\sum_{\alpha\in\nset^d}
\|\varepsilon_{\alpha}\|^2_{\ssy {\bfdot H}^{{\widetilde\xi}(r,\theta)}}\\
&\leq\,C\,h^{2{\widetilde\nu}(r,\theta)}\,\pi^{2{\widetilde\xi}(r,\theta)}
\,\sum_{\alpha\in\nset^d}|\alpha|_{\ssy \nset^d}^{2{\widetilde\xi}(r,\theta)}.\\
\end{split}
\end{equation}
The series in the right hand side of \eqref{soho2} converges if
and only if $-2{\widetilde\xi}(r,\theta)>d$. Thus, in view of
\eqref{ASYM_1}, we arrive at \eqref{soho3} and \eqref{soho4}.
\end{proof}
%
%
%
%
%
\section{Convergence of the Fully-Discrete Approximations}\label{SECTION44}
%
%
\subsection{Consistency estimates}\label{SECTION33}
%
First, we derive some H{\"o}lder-type bounds for ${\widehat u}$.
%
%
%
%
\begin{lemma}\label{HolderT}
Let ${\widehat u}$ be the solution of \eqref{AC2}.
Then, there exist a real positive constant $C$, which is
independent of $T$, $\Delta{t}$ and $\Delta{x}$, such that
\begin{equation}\label{IndianOceanII}
\left\{{\mathbb E}\left[ \left\|\int_{\tau_a}^{\tau_b}
\left[{\widehat u}(\tau_b,\cdot) -{\widehat
u}(\tau,\cdot)\right]\,d\tau\right\|_{\ssy 0,D}^2
\right]\right\}^{\half} \leq\,C
\,\,\,(p_d((\tau_b-\tau_a)^{\frac{1}{4}}))^{\half}
\,\,\,|\tau_b-\tau_a|^{1+\frac{4-d}{8}}
\end{equation}
and
\begin{equation}\label{IndianOcean}
\left\{{\mathbb E}\left[ \|{\widehat u}(\tau_b,\cdot) -{\widehat
u}(\tau_a,\cdot)\|_{\ssy 0,D}^2\right]\right\}^{\half} \leq\,C
\,\,\,(p_d((\tau_b-\tau_a)^{\frac{1}{4}}))^{\half}
\,\,\,|\tau_b-\tau_a|^{\frac{4-d}{8}}
\end{equation}
for $\tau_a$, $\tau_b\in[0,T]$ with $\tau_a\leq\tau_b$.
\end{lemma}
%
%
%
%
%
%
\begin{proof}
We will omit the proof of \eqref{IndianOcean} because it is
similar to that of \eqref{IndianOceanII} which follows.
\par
Let $m\in\{1,\dots,M\}$, $\tau_b\in(0,T]$ and $\tau_a\in[0,T]$
with $\tau_a<\tau_b$, and $\mu(\cdot):=\int_{\tau_a}^{\tau_b}
\left[{\widehat u}(\tau_b,\cdot) -{\widehat
u}(\tau,\cdot)\right]\,d\tau$. First we assume that $\tau_a>0$.
Then, we use \eqref{HatUform}, \eqref{WNEQ2}, the It{\^o}-isometry
property of the stochastic integral, \eqref{GreenKernel} and the
$L^2(D)-$orthogonality of
$(\varepsilon_{\alpha})_{\alpha\in\nset^d}$, to obtain
\begin{equation*}
\begin{split}
{\mathbb E}\left[\|\mu\|_{\ssy 0,D}^2\right]
&=\tfrac{1}{\Delta{t}(\Delta{x})^d}\, \int_{\ssy D}
\Bigg\{\sum_{n\in{\mathcal N}_{\star}}\sum_{\ell\in{\mathcal
J}_{\star}^d}\Big[\,\int_{\ssy S_{n,\ell}}\int_{\tau_a}^{\tau_b}
\big[{\mathcal X}_{(0,\tau_b)}(s')\,G(\tau_b-s';x,y')\\
&\hskip5.5truecm -{\mathcal X}_{(0,\tau)}(s')
\,G(\tau-s';x,y')\big]\,d\tau ds'dy'
\,\Big]^2\Bigg\}\,dx\\
&=\,\tfrac{1}{\Delta{t}(\Delta{x})^d} \sum_{n\in{\mathcal
N}_{\star}} \sum_{\ell\in{\mathcal
J}_{\star}^d}\Bigg\{\,\sum_{\alpha\in\nset^d} \Big(\int_{\ssy
T_n}\int_{\tau_a}^{\tau_b} \big[ {\mathcal X}_{(0,\tau_b)}(s')
\,e^{-\lambda_{\alpha}^2(\tau_b-s')}\\
&\hskip4.5truecm -{\mathcal
X}_{(0,\tau)}(s')\,e^{-\lambda_{\alpha}^2(\tau-s')}\big] \,d\tau
ds'\Big)^2\,\Big(\int_{\ssy D_{\ell}}\varepsilon_{\alpha}(y')
\,dy'\,\Big)^2\,\Bigg\}\\
\end{split}
\end{equation*}
which, along with the use of the Cauchy-Schwarz inequality, yields
\begin{equation}\label{mundo_japan}
\begin{split}
{\mathbb E}\left[\|\mu\|_{\ssy 0,D}^2\right]
&\leq\,\sum_{\alpha\in\nset^d}\Bigg\{\,\int_0^{\ssy T}
\left[\int_{\tau_a}^{\tau_b}\left[ {\mathcal
X}_{(0,\tau_b)}(s')\,e^{-\lambda_{\alpha}^2(\tau_b-s')} -{\mathcal
X}_{(0,\tau)}(s')\,e^{-\lambda_{\alpha}^2(\tau-s')}\right]\,d\tau
\right]^2
\,ds'\Bigg\}\\
%
%
&\leq\,(\tau_b-\tau_a)\, \sum_{\alpha\in\nset^d}\left(
\int_{\tau_a}^{\tau_b}\!\!
\int_0^{\tau}\,\left[e^{-\lambda_{\alpha}^2(\tau_b-s')}
-e^{-\lambda_{\alpha}^2(\tau-s')}\right]^2\,ds'
d\tau\right.\\
&\hskip5.5truecm\left.+\int_{\tau_a}^{\tau_b}\!\!
\int_{\tau}^{\tau_b}
e^{-2\lambda_{\alpha}^2(\tau_b-s')} \,ds' d\tau\right)\\
&\leq \,(\tau_b-\tau_a)^2\ken\sum_{\alpha\in\nset^d}
\tfrac{1-e^{-2\,\lambda_{\alpha}^2\,(\tau_b-\tau_a)}}{\lambda_{\alpha}^2}.
\end{split}
\end{equation}
Finally, we combine \eqref{mundo_japan} and \eqref{ASYM_2} to
arrive at \eqref{IndianOceanII}. The case $\tau_a=0$ follows by
moving along the lines of the proof above using that ${\widehat
u}(0,x)=0$.
\end{proof}
%

\par
Next, we show a consistency result for the Backward Euler
time-discrete approximations of ${\widehat u}$, which is based on
the result of Lemma~\ref{HolderT}.
%
%
\begin{proposition}\label{Synepes1}
Let ${\widehat u}$ be the solution of \eqref{AC2} and $({\widehat
\sigma}_m)_{m=1}^{\ssy M}$ be stochastic functions defined by
\begin{equation}\label{Ydaspis200}
{\widehat u}(\tau_m,\cdot)-{\widehat u}(\tau_{m-1},\cdot)
+k_m\,\Delta^2{\widehat u}(\tau_m,\cdot)
=\int_{\ssy\Delta_m}{\widehat W}\,d\tau +{\widehat
\sigma}_m\quad\text{\rm a.s.}, \quad m=1,\dots,M.
\end{equation}
Then it holds that
\begin{equation}\label{Ydaspis201}
\left(\,{\mathbb E}\left[ \|T_{\ssy B}{\widehat \sigma}_m\|_{\ssy
0,D}^2 \right]\,\right)^{\half} \leq\,C\,\,
\big(p_d(k_m^{\frac{1}{4}})\big)^{\half}\,\,
\,(k_m)^{1+\frac{4-d}{8}},\quad m=1,\dots,M.
\end{equation}
\end{proposition}
%
%
%
\begin{proof}
Let $m\in\{1,\dots,M\}$. Integrating the equation in \eqref{AC2}
over $\Delta_m$ and subtracting it from \eqref{Ydaspis200}, we
conclude that
$T_{\ssy B}{\widehat\sigma}_m(\cdot)=\int_{\ssy\Delta_m}\left[\,
{\widehat u}(\tau_m\cdot)-{\widehat
u}(\tau,\cdot)\,\right]\,d\tau$ a.s..
Thus, to get the bound \eqref{Ydaspis201}, we apply the result
\eqref{IndianOceanII} on the latter equality.
\end{proof}
%
%
%
%
\subsection{Discrete in time
$L^2_t(L^2_{\ssy P}(L^2_x))$ error estimate}
We first obtain a discrete in time $L^2_t(L^2_{\ssy P}(L^2_x))$
error estimate for the Backward Euler fully-discrete
approximations of ${\widehat u}$, by connecting it to the error
estimate of Theorem~\ref{coconut12} for the space-discrete
approximation of ${\widehat u}$.
%
%
%
%
\begin{theorem}\label{Vivo}
Let $r\in\{2,3,4\}$, ${\widehat u}$ be the solution of \eqref{AC2}
and $({\widehat U}_h^m)_{m=0}^{\ssy M}\subset M_h$ be the Backward
Euler fully-discrete approximations of ${\widehat u}$ defined in
\eqref{FullDE1}-\eqref{FullDE2}. Then there exists a constant
$C>0$, independent of $T$, $r$, $\Delta{t}$, $\Delta{x}$, $h$, $M$
and $(k_m)_{m=1}^{\ssy M}$, such that:
\begin{equation}\label{coup1}
\left\{\sum_{m=1}^{\ssy M}k_m\, {\mathbb E}\Big[\|{\widehat U}_h^m
-{\widehat u}(\tau_m,\cdot)\|_{\ssy 0,D}^2\Big]
\right\}^{\frac{1}{2}}
\leq\,C\,\sqrt{T}\,\Big[\,\epsilon^{-\frac{1}{2}}
\,\,\,h^{\nu(r,d)-\epsilon} +{\widehat\omega}(k_{\ssy\rm max})
\,\,\,(k_{\ssy\rm max})^{\frac{4-d}{8}}\,\Big]
\end{equation}
for $\epsilon\in(0,\nu(r,d)]$, where ${\widehat\omega}(k_{\ssy\rm
max}):= (p_d((k_{\ssy\rm max})^{\frac{1}{4}}))^{\half}$ and
$\nu(r,d)$ is defined in \eqref{soho4}.
\end{theorem}
%
%
%
%
%
%
%
%
%
%
%
\begin{proof}
Let ${\widehat u}_h$ be the space-discrete approximation of
${\widehat u}$ defined in \eqref{Semi}, ${\widehat e}={\widehat
u}-{\widehat u}_h$, $z_h^m:={\widehat U}_h^m-{\widehat
u}_h(\tau_m)\in S_h^r$ for $m=0,\dots,M$, and
$V_h:=\left\{\sum_{m=1}^{\ssy M}k_m {\mathbb
E}\left[\|z_h^m\|_{\ssy 0,D}^2\right]\right\}^{\half}$. First, we
observe that
\begin{equation}\label{Greece_China}
\left\{\sum_{m=1}^{\ssy M}k_m\, {\mathbb E}\left[ \|{\widehat
U}_h^m-{\widehat u}(\tau_m,\cdot)\|_{\ssy 0,D}^2
\right]\right\}^{\half} \leq\,V_h+\sqrt{T}\,\,\max_{[0,T]}\left\{
{\mathbb E}\left[ \|{\widehat e}\|_{\ssy 0,D}^2
\right]\right\}^{\half}.
\end{equation}
Integrating \eqref{Semi} over $\Delta_m$ and subtracting the
obtained relation from \eqref{FullDE2}, we arrive at
\begin{equation}\label{ErrorEqIndos1}
T_{\ssy B,h}(z_h^m-z_h^{m-1})+k_m\,z_h^m=\rho_{h,m}
\quad\text{\rm a.s.},\quad m=1,\dots,M,
\end{equation}
where
$\rho_{h,m}:=\int_{\ssy\Delta_m} \big[\,{\widehat u}_h(\tau,\cdot)
-{\widehat u}_h(\tau_m,\cdot)\,\big]\,d\tau$.
Take the $(\cdot,\cdot)_{\ssy 0,D}-$inner product of both sides of
\eqref{ErrorEqIndos1} with $z_h^m$, sum with respect to $m$ from
$1$ up to $M$, to obtain
\begin{equation}\label{ErrorEqIndos2}
\sum_{m=1}^{\ssy M}(\Delta T_{\ssy B,h}z_h^m-\Delta T_{\ssy
B,h}z_h^{m-1}, \Delta T_{\ssy B,h}z_h^m)_{\ssy 0,D}
+\sum_{m=1}^{\ssy M}k_m\|z_h^m\|_{\ssy 0,D}^2 =\sum_{m=1}^{\ssy M}
(\rho_{h,m},z_h^m)_{\ssy 0,D} \quad\text{\rm a.s.}.
\end{equation}
Since $z_h^0=0$, we conclude that
$\sum_{m=1}^{\ssy M}(\Delta T_{\ssy B,h}z_h^m-\Delta T_{\ssy
B,h}z_h^{m-1}, \Delta T_{\ssy B,h}z_h^m)_{\ssy
0,D}\ge\tfrac{1}{2}\, \|\Delta T_{\ssy B,h}z_h^{\ssy M}\|^2_{\ssy
0,D}$ a.s..
Thus, taking expected values in \eqref{ErrorEqIndos2}
and using the Cauchy-Schwarz inequality we get
\begin{equation}\label{Greece_Turk}
\begin{split}
(V_h)^2 \leq&\,{\mathbb E}\left[ \sum_{m=1}^{\ssy M}k_m^{-1}
\,\|\rho_{h,m}\|_{\ssy 0,D}^2 \right]\\
&\leq\,{\mathbb E} \left[\sum_{m=1}^{\ssy M} \,\int_{\ssy
D}\int_{\ssy \Delta_m} \big[{\widehat u}_h(\tau,x)
-{\widehat u}_h(\tau_m,x)\big]^2\,d\tau\,dx\right]\\
&\leq\,\sum_{m=1}^{\ssy M}
\int_{\ssy\Delta_m}
{\mathbb E}\left[\|{\widehat u}_h(\tau,\cdot)
-{\widehat u}_h(\tau_m,\cdot)\|_{\ssy 0,D}^2\right]\,d\tau.\\
\end{split}
\end{equation}
Using \eqref{Greece_Turk} and \eqref{IndianOcean}, we conclude
that
\begin{equation}\label{Nikos_dinner}
\begin{split}
V_h\leq&\left\{\sum_{m=1}^{\ssy M} \int_{\ssy\Delta_m} {\mathbb
E}\left[ \|{\widehat e}(\tau,\cdot) -{\widehat
e}(\tau_m,\cdot)\|_{\ssy 0,D}^2\right]\,d\tau
\right\}^{\frac{1}{2}} +\left\{\sum_{m=1}^{\ssy M}
\int_{\ssy\Delta_m} {\mathbb E}\left[ \|{\widehat u}(\tau,\cdot)
-{\widehat u}(\tau_m,\cdot)\|_{\ssy 0,D}^2\right]\,d\tau
\right\}^{\frac{1}{2}}\\
\leq&\,C\,\sqrt{T}\,\left[\,\max_{[0,T]}\left(\,{\mathbb E}\left[
\|\,{\widehat e}\,\|_{\ssy 0,D}^2\right]\,\right)^{\half}
+(p_d((k_{\ssy\rm max})^{\frac{1}{4}}))^{\frac{1}{2}}
\,\,(k_{\ssy\rm
max})^{\frac{4-d}{8}}\,\right].\\
\end{split}
\end{equation}
Thus, \eqref{coup1} follows from \eqref{Greece_China},
\eqref{Nikos_dinner} and \eqref{soho3}.
\end{proof}
%
%
%
%
%
%
%
%
\subsection{Discrete in time $L^{\infty}_t(L^2_{\ssy P}(L^2_x))$
error estimate}
To get a discrete in time $L^{\infty}_t(L^2_{\ssy P}(L^2_x))$ error
estimate for the Backward Euler fully-discrete approximations of
${\widehat u}$, we compare them to the Backward Euler
time-discrete approximations of ${\widehat u}$ defined in
\eqref{BackE1}--\eqref{BackE2}.
\par
For that we derive first a discrete in time $L^2_t(L^2_x)$ error
estimate between the Backward Euler time-discrete and the Backward
Euler fully discrete approximations of the solution $w$ of
\eqref{Det_Parab} given below:
First set
\begin{equation}\label{DetFD1}
W_h^0:=P_hw_0.
\end{equation}
Then, for $m=1,\dots,M$, find $W_h^m\in M_h$ such that
\begin{equation}\label{DetFD2}
W_h^m-W_h^{m-1} +k_m\,B_hW_h^m =0.
\end{equation}
%
%
%
%
%
\begin{proposition}\label{Aygo_Kokora}
Let $r\in\{2,3,4\}$, $w$ be the solution of the problem
\eqref{Det_Parab}, $(W^m)_{m=0}^{\ssy M}$ be the Backward Euler
time-discrete approxi\-mations of $w$ defined in
\eqref{BEDet1}-\eqref{BEDet2}, and $(W_h^m)_{m=0}^{\ssy M}$ be the
Backward Euler fully-discrete approximations of $w$ specified in
\eqref{DetFD1}-\eqref{DetFD2}. If $w_0\in {\bfdot H}^3(D)$, then,
there exists a constant $C>0$, independent of $T$, $h$, $M$ and
$(k_m)_{m=1}^{\ssy M}$, such that
\begin{equation}\label{tiger_river1}
\left(\,\sum_{m=1}^{\ssy M}k_m\,\|W^m-W_h^m\|^2_{\ssy 0,D}
\,\right)^{\frac{1}{2}}\leq \,C\,\,h^{{\widetilde\nu}(r,\theta)}
\,\,\|w_0\|_{\ssy {\bfdot H}^{{\widetilde\xi}(r,\theta)}}
\quad\forall\,\theta\in[0,1],
\end{equation}
where ${\widetilde\nu}(r,\theta)$ and ${\widetilde\xi}(r,\theta)$
are defined in \eqref{Basilico_4}.
\end{proposition}
%
%
%
%
%
%
%
\begin{proof}
Let $E^m:=W^m-W_h^m$ for $m=0,\dots,M$. We will get
\eqref{tiger_river1} by interpolation, showing it for $\theta=0$
and $\theta=1$.
\par
We use \eqref{BEDet2} and \eqref{DetFD2}, to obtain:
$T_{\ssy B,h}(E^m-E^{m-1})+k_m\,E^m =k_m \,(T_{\ssy B}-T_{\ssy
B,h})\Delta^2W^m$ for $m=1,\dots,M$.
Since $T_{\ssy B,h}E^0=0$, proceeding as in the proof of
Theorem~\ref{Vivo}, it follows that
\begin{equation}\label{citah3}
\sum_{m=1}^{\ssy M}k_m\,\|E^m\|_{\ssy 0,D}^2
\leq\,\sum_{m=1}^{\ssy M}k_m\, \left\|(T_{\ssy B}-T_{{\ssy B},h})
\Delta^2W^m\right\|_{\ssy 0,D}^2.
\end{equation}
\par
Let $r=3$. Then, by \eqref{ARA2} and \eqref{citah3}, we obtain
\begin{equation}\label{citah4}
\sum_{m=1}^{\ssy M}k_m\,\|E^m\|_{\ssy 0,D}^2
\leq\,C\,h^{8}\,\sum_{m=1}^{\ssy M}k_m\,
\left\|\Delta^2W^m\right\|_{\ssy 0,D}^2.
\end{equation}
Taking the $(\cdot,\cdot)_{\ssy 0,D}-$inner product of
\eqref{BEDet2} with $\Delta^2W^m$, and then integrating by parts
and summing with respect to $m$ from $1$ up to $M$, it follows
that
\begin{equation}\label{citah5}
\sum_{m=1}^{\ssy M} (\Delta W^m-\Delta W^{m-1},\Delta W^m)_{\ssy
0,D} +\sum_{m=1}^{\ssy M} k_m\,\|\Delta^2W^m\|_{\ssy 0,D}^2=0.
\end{equation}
Since $\sum_{m=1}^{\ssy M}\big(\Delta W^m-\Delta W^{m-1}, \Delta
W^m\big)_{\ssy 0,D}\ge\tfrac{1}{2}\, \big(\,\|\Delta W^{\ssy
M}\|^2_{\ssy 0,D} -\|\Delta W^0\|_{\ssy 0,D}^2\,\big)$,
\eqref{citah5} yields
\begin{equation}\label{citah6}
\sum_{m=1}^{\ssy M}k_m\,\|\Delta^2W^m\|_{\ssy 0,D}^2
\leq\,\tfrac{1}{2}\,\|w_0\|_{\ssy 2,D}^2.
\end{equation}
Combining, now, \eqref{citah4}, \eqref{citah6} and
\eqref{H_equiv}, we obtain
\begin{equation}\label{citah7}
\left(\,\sum_{m=1}^{\ssy M}k_m\,\|E^m\|_{\ssy
0,D}^2\,\right)^{\frac{1}{2}}
\leq\,C\,h^{4}\,\|w_0\|_{\ssy {\bfdot H}^2}.
\end{equation}
\par
Let $r=2$. Then, by \eqref{ARA2}, \eqref{minus_equiv} and
\eqref{citah3}, we obtain
\begin{equation}\label{citah8}
\begin{split}
\sum_{m=1}^{\ssy M}k_m\,\|E^m\|_{\ssy 0,D}^2
\leq&\,C\,h^{4}\,\sum_{m=1}^{\ssy M}k_m\,
\left\|\Delta^2W^m\right\|_{\ssy {\bfdot H}^{-1}}^2\\
\leq&\,C\,h^{4}\,\left[-\sum_{m=1}^{\ssy M}k_m\,
(T_{\ssy E}\Delta^2W^m,\Delta^2W^m)_{\ssy 0,D}\right]\\
\leq&\,C\,h^{4}\,\left[-\sum_{m=1}^{\ssy M}k_m\,
(\Delta W^m,\Delta^2W^m)_{\ssy 0,D}\right].\\
\end{split}
\end{equation}
Taking the $(\cdot,\cdot)_{\ssy 0,D}-$inner product of
\eqref{BEDet2} with $\Delta W^m$, integrating by parts and summing
with respect to $m$ from $1$ up to $M$, it follows that
\begin{equation}\label{citah9}
\sum_{m=1}^{\ssy M}\big(\nabla W^m-\nabla W^{m-1}, \nabla
W^m\big)_{\ssy 0,D} -\sum_{m=1}^{\ssy M} k_m\,(\Delta^2 W^m,\Delta
W^m)_{\ssy 0,D}=0.
\end{equation}
Since
$ \sum_{m=1}^{\ssy M}(\nabla W^m-\nabla W^{m-1}, \nabla W^m)_{\ssy
0,D} \ge\,\tfrac{1}{2}\,\big[\, \|\nabla W^{\ssy M}\|^2_{\ssy 0,D}
-\|\nabla W^0\|_{\ssy 0,D}^2\,\bigr], $
\eqref{citah9} yields
\begin{equation}\label{citah10}
-\sum_{m=1}^{\ssy M}k_m\,(\Delta^2 W^m,\Delta W^m)_{\ssy 0,D}
\leq\,\tfrac{1}{2}\,\|w_0\|_{\ssy 1,D}^2.
\end{equation}
Combining \eqref{citah8}, \eqref{citah10} and
\eqref{H_equiv} we get
\begin{equation}\label{citah11}
\left(\,\sum_{m=1}^{\ssy M}k_m\,\|E^m\|_{\ssy
0,D}^2\,\right)^{\frac{1}{2}}
\leq\,C\,h^2\,\|w_0\|_{\ssy {\bfdot H}^1}.
\end{equation}
\par
Let $r=4$. Then, observing that $\Delta^2 W^m\in {\bfdot H}^1(D)$
and using the relations \eqref{ARA2}, \eqref{minus_equiv} and
\eqref{citah3}, we obtain
\begin{equation}\label{citah_c1}
\begin{split}
\sum_{m=1}^{\ssy M}k_m\,\|E^m\|_{\ssy 0,D}^2
\leq&\,C\,h^{10}\,\sum_{m=1}^{\ssy M}k_m\,
\left\|\Delta^2W^m\right\|_{\ssy {\bfdot H}^{1}}^2\\
\leq&\,C\,h^{10}\,\sum_{m=1}^{\ssy M}k_m\,
\left\|\Delta^3W^m\right\|_{\ssy {\bfdot H}^{-1}}^2\\
\leq&\,C\,h^{10}\,\left[-\sum_{m=1}^{\ssy M}k_m\,
(T_{\ssy E}\Delta^3W^m,\Delta^3W^m)_{\ssy 0,D}\right]\\
\leq&\,C\,h^{10}\,\left[-\sum_{m=1}^{\ssy M}k_m\,
(\Delta^2 W^m,\Delta^3W^m)_{\ssy 0,D}\right].\\
\end{split}
\end{equation}
After, Applying the operator $\Delta$ on \eqref{BEDet2}, take the
$(\cdot,\cdot)_{\ssy 0,D}-$inner product of the obtained relation
with $\Delta^2 W^m$, integrate by parts and sum with respect to
$m$ from $1$ up to $M$, to get
\begin{equation}\label{citah_c2}
-\sum_{m=1}^{\ssy M}\big(\Delta W^m-\Delta W^{m-1}, \Delta^2
W^m\big)_{\ssy 0,D} -\sum_{m=1}^{\ssy M} k_m\,(\Delta^3
W^m,\Delta^2 W^m)_{\ssy 0,D}=0.
\end{equation}
Also, we have
\begin{equation}\label{citah_c22}
\begin{split}
-\sum_{m=1}^{\ssy M}(\Delta W^m-\Delta W^{m-1}, \Delta^2
W^m)_{\ssy 0,D}\ge\,&\sum_{m=1}^{\ssy M}\left(\,\|\Delta
W^m\|_{\ssy {\bfdot H}^1}^2-\|\Delta W^m\|_{\ssy{\bfdot H}^1}
\,\|\Delta
W^{m-1}\|_{\ssy{\bfdot H}^1}\,\right)\\
\ge\,&\tfrac{1}{2}\,\left(\, \|\Delta W^{\ssy M}\|^2_{\ssy{\bfdot
H}^1}
-\|\Delta W^0\|_{\ssy{\bfdot H}^1}\,\right).\\
\end{split}
\end{equation}
Thus, \eqref{citah_c2} and \eqref{citah_c22} yield
\begin{equation}\label{citah_c3}
-\sum_{m=1}^{\ssy M}k_m\,(\Delta^3 W^m,\Delta^2 W^m)_{\ssy 0,D}
\leq\,\tfrac{1}{2}\,\|w_0\|_{\ssy{\bfdot H}^3}^2.
\end{equation}
Combining \eqref{citah_c1} and \eqref{citah_c3} we get
\begin{equation}\label{citah_c4}
\left(\,\sum_{m=1}^{\ssy M}k_m\,\|E^m\|_{\ssy
0,D}^2\,\right)^{\frac{1}{2}}
\leq\,C\,h^5\,\|w_0\|_{\ssy {\bfdot H}^3}.
\end{equation}
%
%
%
Thus, the relations \eqref{citah7}, \eqref{citah11} and
\eqref{citah_c4} yield \eqref{tiger_river1} for $\theta=1$.
\par
Since $T_{{\ssy B},h}(W_h^m-W_h^{m-1})+k_m\,W_h^m=0$ for
$m=1,\dots,M$, we obtain
\begin{equation*}
(\Delta T_{\ssy B,h}W_h^m-\Delta T_{\ssy B,h}W_h^{m-1},\Delta
T_{\ssy B,h}W^m_h)_{\ssy 0,D} +k_m\,\|W_h^m\|_{\ssy 0,D}^2,\quad
m=1,\dots,M,
\end{equation*}
which, along with \eqref{PanwFragma1} and \eqref{minus_equiv},
yields
\begin{equation}\label{citah_c111}
\begin{split}
\sum_{m=1}^{\ssy M} k_m\,\|W_h^m\|_{\ssy 0,D}^2
\leq&\,\tfrac{1}{2}\,\|\Delta T_{\ssy B,h}w_0\|_{\ssy 0,D}^2\\
\leq&\,C\,\|w_0\|_{\ssy {\bfdot H}^{-2}}.\\
\end{split}
\end{equation}
Now, using \eqref{BEDet2} and \eqref{TB-prop1}, we obtain
$(T_{\ssy E}W^m-T_{\ssy E}W^{m-1},T_{\ssy E}W^m)_{\ssy 0,D}
+k_m\,\|W^m\|_{\ssy 0,D}^2=0$ for $m=1,\dots,M$, which yields
$\|T_{\ssy E}W^m\|_{\ssy 0,D}^2 -\|T_{\ssy E}W^{m-1}\|_{\ssy
0,D}^2 +2\,k_m\,\|W^m\|_{\ssy 0,D}^2\leq 0$ for $m=1,\dots,M$.
Then, summing with respect to $m$ from $1$ up to $M$, and using
\eqref{ElReg1} and \eqref{minus_equiv} we obtain
\begin{equation}\label{Ydaspis904}
\begin{split}
\sum_{k=1}^{\ssy M}k_m\|W^m\|_{\ssy 0,D}^2
&\leq\,\tfrac{1}{2}\,\|T_{\ssy E}w^0\|_{\ssy 0,D}^2\\
&\leq\,C\,\|w_0\|_{\ssy{\bfdot H}^{-2}}^2.\\
\end{split}
\end{equation}
Finally, combine \eqref{citah_c111} with \eqref{Ydaspis904} to get
$\left(\,\sum_{m=1}^{\ssy M}k_m\,\|E^m\|_{\ssy
0,D}^2\,\right)^{\frac{1}{2}}\leq\,C\,\|w_0\|_{\ssy {\bfdot
H}^{-1}}$,
which is equivalent to \eqref{tiger_river1} for $\theta=0$.
\end{proof}
%
%
%
%
\par
The following lemma ensures the existence of a continuous
Green function for the solution operator of a discrete
elliptic problem.
%
%
\begin{lemma}\label{prasinolhmma}
Let $r\in\{2,3,4\}$, $\epsilon>0$, $f\in L^2(D)$ and $\psi_h\in
M_h$ such that
\begin{equation}\label{bohqos_pro}
\epsilon\,B_h\psi_h+\psi_h=P_hf.
\end{equation}
Then there exists a function $G_{h,\epsilon}\in C({\overline {D\times D}})$
such that
\begin{equation}\label{prasinogreen}
\psi_h(x)=\int_{\ssy D} G_{h,\epsilon}(x,y)\,f(y)\,dy
\quad\forall\,x\in{\overline D}
\end{equation}
and $G_{h,\epsilon}(x,y)=G_{h,\epsilon}(y,x)$ for $x,y\in {\overline D}.$
\end{lemma}
%
%
%
%
\begin{proof}
Keeping the notation and the constructions of the proof of
Lemma~\ref{prasino}, we conclude that there are
$(\mu_j)_{j=1}^{n_h}\subset\rset$ such that
$\psi_h=\sum_{j=1}^{n_h}\mu_j\,\chi_j$. Thus, \eqref{bohqos_pro}
is equivalent to $\mu_i
=\frac{1}{1+\epsilon\lambda_{h,i}}\,(f,\chi_i)_{\ssy 0,D}$ for
$i=1,\dots,n_h$. Finally, we obtain \eqref{prasinogreen} with
$G_{h,\epsilon}(x,y)=\sum_{j=1}^{n_h}
\frac{\chi_j(x)\chi_j(y)}{1+\epsilon\lambda_{h,j}}$.
\end{proof}
\par
We are ready to compare, in the discrete in time
$L^{\infty}_t(L^2_{\ssy P}(L^2_x))$ norm, the time-discrete with
the fully-discrete Backward Euler approximations of ${\widehat
u}$.
%
%
%
%
%
%
\begin{proposition}\label{Tigrakis}
Let $r\in\{2,3,4\}$, ${\widehat u}$ be the solution of the problem
\eqref{AC2}, $({\widehat U}_h^m)_{m=0}^{\ssy M}$ be the Backward
Euler fully-discrete approximations of ${\widehat u}$ specified in
\eqref{FullDE1}-\eqref{FullDE2}, and $({\widehat U}^m)_{m=0}^{\ssy
M}$ be the Backward Euler time-discrete approximations of
${\widehat u}$ specified in \eqref{BackE1}-\eqref{BackE2}.
If the partition $(\tau_m)_{m=0}^{\ssy M}$ is uniform, i.e.
$k_m=\Delta\tau$ for $m=1,\dots,M$, then, there exists a constant
$C>0$, independent of $\Delta{x}$, $\Delta{t}$, $h$, $M$ and
$\Delta\tau$,  such that
\begin{equation}\label{Lasso1}
\max_{1\leq{m}\leq {\ssy M}}\left\{{\mathbb E}\left[
\big\|{\widehat U}_h^m -{\widehat U}^m\big\|^2_{\ssy
0,D}\right]\right\}^{\half} \leq\,C\,\epsilon^{-\frac{1}{2}}
\,\,\,h^{\nu(r,d)-\epsilon},\quad\forall\,\epsilon\in(0,\nu(r,d)]
\end{equation}
where $\nu(r,d)$ has been defined in \eqref{soho4}.
\end{proposition}
%
%
%
%
%
%
%
%
%
\begin{proof}
Let $I:L^2(D)\to L^2(D)$ be the identity operator and
$\Lambda_h:L^2(D)\to S^r_h$ be the inverse discrete elliptic
operator given by $\Lambda_h:=(I+\Delta\tau\,B_h)^{-1}P_h$ and
having a Green function $G_{h,{\ssy\Delta\tau}}$ (cf.
Lemma~\ref{prasinolhmma}). Also, for $\ell\in\nset$, we denote by
$G_{h,{\ssy\Delta\tau},\ell}$ the Green function of
$\Lambda_h^{\ell}$.
Using, now, an induction argument, from \eqref{FullDE2} we
conclude that
${\widehat U}_h^m=\sum_{j=1}^{\ssy m} \int_{\ssy\Delta_j}
\Lambda_h^{m-j+1}{\widehat W}(\tau,\cdot)\,d\tau$, $m=1,\dots,M$,
which is written, equivalently, as follows:
\begin{equation}\label{Anaparastash2}
{\widehat U}_h^m(x)=\int_0^{\tau_m}\!\!\!\int_{\ssy D}
\,{\widehat{\mathcal D}}_{h,m}(\tau;x,y)\,{\widehat W}(\tau,y)
\,dyd\tau\quad\forall\,x\in{\overline D}, \ken m=1,\dots,M,
\end{equation}
where
\begin{equation*}
{\widehat{\mathcal D}}_{h,m}(\tau;x,y)
:=\sum_{j=1}^m{\mathcal X}_{\ssy\Delta_j}(\tau)
\,G_{h,{\ssy\Delta\tau},m-j+1}(x,y)\quad\forall
\,\tau\in[0,T],\ken\forall\,x,y\in D.
\end{equation*}
Using \eqref{Anaparastash1}, \eqref{Anaparastash2}, the
It{\^o}-isometry property of the stochastic integral,
\eqref{HSxar} and the Cauchy-Schwarz inequality, we get
\begin{equation*}
\begin{split}
{\mathbb E}\left[
\|{\widehat U}^m-{\widehat U}_h^m\|_{\ssy 0,D}^2
\right]
&\leq\int_0^{\tau_m}\Big(\int_{\ssy D}\!\int_{\ssy D}
\,\big[{\widehat{\mathcal K}}_m(\tau;x,y)
-{\widehat{\mathcal D}}_{h,m}(\tau;x,y)\big]^2
\,dydx\Big)\,d\tau\\
&\leq\, \sum_{j=1}^m\int_{\ssy\Delta_{j}}\,\|\Lambda^{m-j+1}
-\Lambda_h^{m-j+1}\|_{\ssy\rm HS}^2\,d\tau,
\quad m=1,\dots,M,\\
\end{split}
\end{equation*}
where $\Lambda$ is the inverse elliptic operator defined in the
proof of Theorem~\ref{TimeDiscreteErr1}. Now, we use the
definition of the Hilbert-Schmidt norm and the deterministic error
estimate \eqref{tiger_river1}, to have
\begin{equation*}
\begin{split}
%
%
{\mathbb E}\left[ \|{\widehat U}^m-{\widehat U}_h^m\|_{\ssy 0,D}^2
\right] &\leq\,\sum_{j=1}^m\,\Delta\tau
\left[\,\sum_{\alpha\in\nset^d}
\|\Lambda^{m-j+1}\varepsilon_{\alpha}
-\Lambda^{m-j+1}_h\varepsilon_{\alpha}\|^2_{\ssy 0,D}
\,\right]\\
%
%
&\leq\,\sum_{\alpha\in\nset^d}\left[\, \sum_{j=1}^m\,\Delta\tau\,
\|\Lambda^j\varepsilon_{\alpha}
-\Lambda_h^j\varepsilon_{\alpha}\|^2_{\ssy 0,D}
\,\right]\\
&\leq\,C\,h^{2{\widetilde\nu}(r,\theta)}
\,\sum_{\alpha\in\nset^d}\|\varepsilon_{\alpha}\|^2_{\ssy {\bfdot
H}^{{\widetilde\xi}(r,\theta)}},\quad m=1,\dots,M,
\quad\forall\,\theta\in[0,1].\\
\end{split}
\end{equation*}
Thus, we arrive at
\begin{equation*}
\max_{1\leq{m}\leq{\ssy M}} {\mathbb E}\left[ \|{\widehat
U}^m-{\widehat U}_h^m\|_{\ssy 0,D}^2 \right]
\leq\,C\,h^{2{\widetilde\nu}(r,\theta)}\,\sum_{\alpha\in\nset^d}
|\alpha|_{\ssy \nset^d}^{2{\widetilde\xi}(r,\theta)},
\quad\forall\,\theta\in[0,1],
\end{equation*}
from which, requiring $-2\,{\widetilde\xi}(r,\theta)>d$,
\eqref{Lasso1}, easily, follows (cf. Theorem~\ref{coconut12}).
\end{proof}
%
%
%
%
%
%
%
\par
The available error estimates allow us to conclude a discrete in
time $L^{\infty}_t(L^2_{\ssy P}(L^2_x))$ convergence of the
Backward Euler fully-discrete approximations of ${\widehat u}$,
over a uniform partition of $[0,T]$.
%
%
%
%
\begin{theorem}\label{FFQEWR}
Let $r\in\{2,3,4\}$, $\nu(r,d)$ be defined by \eqref{soho4},
${\widehat u}$ be the solution of problem \eqref{AC2}, and
$({\widehat U}_h^m)_{m=0}^{\ssy M}$ be the Backward Euler
fully-discrete approximations of ${\widehat u}$ constructed by
\eqref{FullDE1}-\eqref{FullDE2}.
If the partition $(\tau_m)_{m=0}^{\ssy M}$ is uniform, i.e.,
$k_m=\Delta\tau$ for $m=1,\dots,M$, then, there exists a constant
$C>0$, independent of $T$, $h$, $\Delta\tau$, $\Delta{t}$ and
$\Delta{x}$, such that
\begin{equation*}
\max_{0\leq{m}\leq{\ssy M}}\left\{{\mathbb E}\left[ \|{\widehat
U}_h^m-{\widehat u}(\tau_m,\cdot)\|_{\ssy 0,D}^2\right]
\right\}^{\half} \leq\,C\,\left[
\,{\widetilde\omega}(\Delta\tau,\epsilon_1)
\,\,\Delta\tau^{\frac{4-d}{8}-\epsilon_1}
+\epsilon_2^{-\frac{1}{2}}\,h^{\nu(r,d)-\epsilon_2}\,\right],
\end{equation*}
for $\epsilon_1\in\big(0,\frac{4-d}{8}\big]$ and
$\epsilon_2\in\big(0,\nu(r,d)\big]$ where
${\widetilde\omega}(\Delta\tau,\epsilon_1):=\epsilon_1^{-\frac{1}{2}}
+(\Delta\tau)^{\epsilon_1}
(p_d(\Delta\tau^{\frac{1}{4}}))^{\half}$.
\end{theorem}
%
%
%
%
%
%
%
%
%
%
%
\begin{proof}
The estimate is a simple consequence of the error bounds
\eqref{Lasso1} and \eqref{ElPasso}.
\end{proof}
%
%
%
%
%
%
%
%
%
\def\cprime{$'$}
\def\cprime{$'$}
%

%
%
\appendix
%
%
%
\section{ }\label{Gatos_Rex_1}
\par\noindent
\begin{proof}[Proof of \eqref{ASYM_1}]
Let $\epsilon\in(0,2]$. First, we observe that
\begin{equation*}
\begin{split}
\sum_{n=1}^{\infty}\tfrac{1}{n^{1+c_{\star}\epsilon}}
\leq&\,1+\int_1^{+\infty}\tfrac{1}{x^{1+c_{\star}\epsilon}}\;dx\\
\leq&\,\left(2+\tfrac{1}{c_{\star}}\right)\tfrac{1}{\epsilon},
\end{split}
\end{equation*}
which easily yields \eqref{ASYM_1} for $d=1$. For $d=2$, we have
\begin{equation*}
\begin{split}
\sum_{\alpha\in\nset^2}|\alpha|_{\ssy\nset^2}^{-(2+c_{\star}\epsilon)}
\leq&\,2\,\sum_{n=1}^{\infty}\tfrac{1}{(1+n^2)^{\frac{2+c_{\star}\epsilon}{2}}}
+\int_{(1,+\infty)^2}\,|x|^{-(2+c_{\star}\epsilon)}\;dx\\
\leq&\,2\,\sum_{n=1}^{\infty}\tfrac{1}{n^{1+c_{\star}\epsilon}}
+\int_1^{\infty}\!\!\int_0^{\frac{\pi}{2}}\,r^{-(1+c_{\star}\epsilon)}\;drd\theta\\
\leq&\,C\,\epsilon^{-1}.\\
\end{split}
\end{equation*}
For $d=3$, using \eqref{ASYM_1} for $d=2$, we proceed similarly as
follows
\begin{equation*}
\begin{split}
\sum_{\alpha\in\nset^3}|\alpha|_{\ssy\nset^3}^{-(3+c_{\star}\epsilon)}
\leq&\,3\,\sum_{\beta\in\nset^2}(1+|\beta|_{\ssy\nset^2}^2)^{-\frac{3+c_{\star}\epsilon}{2}}
+\int_{(1,+\infty)^3}\,|x|^{-(3+c_{\star}\epsilon)}\;dx\\
\leq&\,3\,\sum_{\beta\in\nset^2}|\beta|_{\ssy\nset^2}^{-(2+c_{\star}\epsilon)}
+\int_1^{\infty}\!\!\int_0^{\frac{\pi}{2}}\!\!\int_0^{\frac{\pi}{2}}
\,\sin(\theta)\,r^{-(1+c_{\star}\epsilon)}\;drd\theta d\phi\\
\leq&\,C\,\epsilon^{-1}.\\
\end{split}
\end{equation*}
\end{proof}
%
%
%
%
%
\section{ }\label{Gatos_Rex_2}
\par\noindent
\begin{proof}[Proof of \eqref{ASYM_2}]
First, we recall from \cite{KZ2008} that
$\sum_{k=1}^{\infty}
\tfrac{1-e^{-k^4\pi^4\delta}}{k^4\pi^4}\leq\,C\,(1
+\delta^{\frac{1}{4}})\,\delta^{\frac{3}{4}}$.
For $d=2$, using the latter inequality, we have
\begin{equation*}
\begin{split}
\sum_{\alpha\in\nset^2}
\tfrac{1-e^{-\lambda_{\alpha}^2\,\delta}}{\lambda_{\alpha}^2}\leq&\,
2\,\sum_{n=1}^{\infty}\tfrac{1-e^{-\pi^4(1+k^2)^2\delta}}{\pi^4(1+k^2)^2}
+\int_{(1,+\infty)^2}
\tfrac{1-e^{-\pi^4|x|^4\,\delta}}{\pi^4\,|x|^4}\;dx\\
\leq&2\,\sum_{n=1}^{\infty}\tfrac{1-e^{-4\pi^4k^4\delta}}{\pi^4k^4}
+\int_0^{\frac{\pi}{2}}\!\int_1^{+\infty}
\tfrac{1-e^{-\pi^4r^4\delta}}{\pi^4\,r^3}\;dr d\theta\\
\leq&\,C\,(1+\delta^{\frac{1}{4}})\,\delta^{\frac{3}{4}}
+\tfrac{\delta^{\frac{1}{2}}}{2\pi}\,\int_0^{+\infty}
\tfrac{1-e^{-z^4}}{z^3}\;dz\\
\leq&\,C\,(1+\delta^{\frac{1}{4}})\,\delta^{\frac{3}{4}}
+\tfrac{\delta^{\frac{1}{2}}}{2\pi}\,\int_0^{+\infty}
e^{-z^3}\;dz,\\
\end{split}
\end{equation*}
which yields
$\sum_{\alpha\in\nset^2}
\tfrac{1-e^{-\lambda_{\alpha}^2\,\delta}}{\lambda_{\alpha}^2}\leq
\,C\,(1+\delta^{\frac{1}{4}}+\delta^{\frac{1}{2}})\,\delta^{\frac{1}{2}}$.
Finally, when $d=3$, using \eqref{ASYM_2} for $d=2$, we obtain
\begin{equation*}
\begin{split}
\sum_{\alpha\in\nset^3}
\tfrac{1-e^{-\lambda_{\alpha}^2\delta}}{\lambda_{\alpha}^2}\leq&\,
3\,\sum_{\beta\in\nset^2}
\tfrac{1-e^{-\pi^4(1+|\beta|_{\nset^2}^2)^2\delta}}
{\pi^4(1+|\beta|_{\nset^2}^2)^2} +\int_{(1,+\infty)^3}
\tfrac{1-e^{-\pi^4|x|^4\,\delta}}{\pi^4\,|x|^4}\;dx\\
\leq&3\,\sum_{\beta\in\nset^2}
\tfrac{1-e^{-4\lambda_{\beta}^2\delta}} {\lambda_{\beta}^2}
+\int_0^{\frac{\pi}{2}}
\!\int_0^{\frac{\pi}{2}}\!\int_0^{+\infty}\,\sin(\phi)\,
\tfrac{1-e^{-\pi^4r^4\delta}}{\pi^4\,r^2}\;dr d\theta d\phi\\
\leq&\,C\,(1+\delta^{\frac{1}{4}}+\delta^{\frac{1}{2}})\,\delta^{\frac{1}{2}}
+\tfrac{\delta^{\frac{1}{4}}}{2\pi^2}\,\int_0^{+\infty}
\tfrac{1-e^{-2\,z^4}}{z^2}\;dz\\
\leq&\,C\,(1+\delta^{\frac{1}{4}}+\delta^{\frac{1}{2}})\,\delta^{\frac{1}{2}}
+\tfrac{\delta^{\frac{1}{4}}}{2\pi^2}\,\int_0^{+\infty}
e^{-z^2}\;dz,\\
\end{split}
\end{equation*}
which yields
$\sum_{\alpha\in\nset^3}
\tfrac{1-e^{-\lambda_{\alpha}^2\,\delta}}{\lambda_{\alpha}^2}\leq
\,C\,(1+\delta^{\frac{1}{4}}+\delta^{\frac{1}{2}}
+\delta^{\frac{3}{4}})\,\delta^{\frac{1}{4}}$.
\end{proof}
%
%
%
%
%
\end{document}